\documentclass[12pt,notitlepage]{amsart}
\usepackage{latexsym,amsfonts,amssymb,amsmath,amsthm}
\usepackage{color}
\usepackage[inner=1.0in,outer=1.0in,bottom=1.0in, top=1.0in]{geometry}

\newcommand{\mz}{\ensuremath{\mathbb Z}}
\newcommand{\mr}{\ensuremath{\mathbb R}}
\newcommand{\mh}{\ensuremath{\mathbb H}}

\newcommand{\shortmod}{\ensuremath{\negthickspace \negthickspace \negthickspace \pmod}}

\newcommand{\intR}{\int_{-\infty}^{\infty}}

\newcommand{\sumstar}{\sideset{}{^*}\sum}
\newcommand{\sumprime}{\sideset{}{'}\sum}

\newcommand{\e}[2]{e\left(\frac{#1}{#2}\right)}

\def\polhk#1{\setbox0=\hbox{#1}{\ooalign{\hidewidth
    \lower1.5ex\hbox{`}\hidewidth\crcr\unhbox0}}}

\theoremstyle{plain}
	\newtheorem{mytheo}{Theorem} [section]
	
	\newtheorem{myprop}[mytheo]{Proposition}
	\newtheorem{mycoro}[mytheo]{Corollary}
     \newtheorem{mylemma}[mytheo]{Lemma}
	\newtheorem{mydefi}[mytheo]{Definition}
	
	\newtheorem{myremark}[mytheo]{Remark}

\theoremstyle{remark}

\numberwithin{equation}{section}
\begin{document}

 \author{Matthew P. Young}
 \address{Department of Mathematics \\
 	  Texas A\&M University \\
 	  College Station \\
 	  TX 77843-3368 \\
 		U.S.A.}
 \email{myoung@math.tamu.edu}
 \thanks{This material is based upon work supported by the National Science Foundation under agreement No. DMS-1401008.  Any opinions, findings and conclusions or recommendations expressed in this material are those of the authors and do not necessarily reflect the views of the National Science Foundation.  
 }

\begin{abstract} 
We analyze certain bilinear forms involving $GL_3$ Kloosterman sums.  As an application, we obtain an improved estimate for the $GL_3$ spectral large sieve inequality.
\end{abstract}
 \title{Bilinear forms with $GL_3$ Kloosterman sums and the spectral large sieve}
\maketitle
\section{Introduction}
Given a family of $L$-functions, $\{ L(s,f) : f \in \mathcal{F} \}$, one of the most basic questions one can study is its orthogonality properties.  More precisely, if $L(s,f) = \sum_{n=1}^{\infty} \lambda_f(n) n^{-s}$, then one wishes to understand $\Delta_{\mathcal{F}}(m,n) :=\sum_{f \in \mathcal{F}} \lambda_f(m) \overline{\lambda_f(n)}$.  
For instance, when the family consists of Dirichlet characters, a formula for $\Delta_{\mathcal{F}}$ is given by orthogonality of characters.  For families of $GL_2$ forms, $\Delta_{\mathcal{F}}$ can be expanded into a sum of Kloosterman sums, by the Petersson/Bruggeman-Kuznetsov trace formula, which has seen extensive applications in number theory.

A large sieve inequality takes this analysis even futher, by bounding
\begin{equation}
\sum_{f \in \mathcal{F}} \Big| \sum_{n \leq N} a_n \lambda_f(n) \Big|^2,
\end{equation}
where $a_n$ are arbitrary complex coefficients.  By general principles, the best one may hope for is a bound of the form $(|\mathcal{F}| + N ) \sum_{n \leq N} |a_n|^2$.  One can view this as a much more robust form of orthogonality, probing the sequence of values of $\lambda_f(n)$ by correlations with arbitrary sequences $a_n$.
Large sieve inequalities are flexible and powerful estimates for bilinear forms having many applications.  For instance, the classical large sieve inequality for Dirichlet characters plays a key role in proving the Bombieri-Vinogradov theorem.  The $GL_2$ spectral large sieve has been valuable in understanding mean values of $L$-functions (in particular, to the fourth moment of the zeta function, which was Iwaniec's original application \cite{IwaniecLargeSieve}).  The reader is referred to \cite[Chapter 7]{IK} for a good introduction to large sieve inequalities.

The corresponding studies of higher rank families are still in their infancy.  
Bump, Friedberg, and Goldfeld \cite{BFG} developed many of the foundational properties of the $GL_3$ Poincare series, and in particular discovered the analogous sums to the $GL_2$ Kloosterman sums.  Recently, Blomer \cite{Blomer} 
succeeded in formulating a $GL_3$ Bruggeman-Kuznetsov formula with smooth bump functions appearing on the spectral side.  Blomer also derived a form of the spectral $GL_3$ large sieve inequality, but without a focus on obtaining a sharp result.  In principle, one may also derive a large sieve inequality from Goldfeld-Kontorovich's work \cite{GK}, but again this was not the focus of the authors and the result would not be numerically strong.

One of our main goals here is to obtain a stronger form of the $GL_3$ spectral large sieve inequality.  To state the results, we set up some of the necessary notation as in \cite{BFG} \cite{Goldfeld} \cite{Blomer}.  Consider the family of Hecke-Maass cusp forms $\phi_j$ for $SL_3(\mz) \backslash \mathcal{H}$, with spectral parameters $\nu_1, \nu_2$.    The Langlands parameters  associated to $\phi_j$ are $\alpha_1 = 2 \nu_1 + \nu_2$, $\alpha_2 = -\nu_1 + \nu_2$, and $\alpha_3 = -\nu_1 - 2 \nu_2$.  Blomer has shown that the number of $\phi_j$ with $\nu_1 = iT_1 + O(1)$, $\nu_2 = iT_2 + O(1)$, weighted by $R_j^{-1}$, where
\begin{equation}
\label{eq:Rjdef}
 R_j = \text{Res}_{s=1} L(s, \phi_j \times \overline{\phi_j}),
\end{equation}
is $\asymp T_1 T_2 (T_1 + T_2)$ (also see \cite[(1.4)]{Blomer}).  This is a natural weighting from the point of view of the Bruggeman-Kuznetsov formula.  Let $\lambda_j(m,n)$ denote the Hecke eigenvalues of $\phi_j$, with $\lambda_j(1,1) = 1$.  With an appropriate choice of scaling of Whittaker functions, then $\| \phi_j \|^2 \asymp R_j$ (e.g., see \cite[Lemma 1]{Blomer}).

\begin{mytheo}
\label{thm:spectralsumlocal}
 For an arbitrary complex sequence $a_n$, we have
 \begin{equation}
 \label{eq:spectralsumlocal}
  \sum_{\substack{\nu_1 = iT_1 + O(1) \\ \nu_2 = iT_2 + O(1)}} \frac{1}{R_j} \Big| \sum_{ n \leq N} a_n \lambda_j(n,1)
\Big|^2 \ll
\Big(T_1 T_2 (T_1 + T_2) +   T_1 T_2 N^{3/2} \Big)^{1+\varepsilon}
\sum_{n \leq N} |a_n|^2.
 \end{equation}
\end{mytheo}
For comparison, Blomer's proof of the spectral large sieve (implicitly) shows
\begin{equation}
\label{eq:BlomerBoundLocal}
  \sum_{\substack{\nu_1 = iT_1 + O(1) \\ \nu_2 = iT_2 + O(1)}} \frac{1}{R_j} \Big| \sum_{ n \leq N} a_n \lambda_j(n,1)
\Big|^2 \ll
\Big(T_1 T_2 (T_1 + T_2) +   T_1 T_2 N^{2} \Big)^{1+\varepsilon}
\sum_{n \leq N} |a_n|^2,
 \end{equation}
so Theorem \ref{thm:BilinearKloosterman} saves a potentially rather large factor $N^{1/2}$.   In fact, Blomer shows a dyadic bound:
\begin{equation}
\label{eq:BlomerBoundDyadic}
  \sum_{\substack{T_1 \leq |\nu_1| \leq 2T_1  \\ T_2 \leq |\nu_2| \leq 2T_2}} \frac{1}{R_j} \Big| \sum_{ n \leq N} a_n \lambda_j(n,1)
\Big|^2 \ll
\Big(T_1^2 T_2^2 (T_1 + T_2) +   T_1 T_2 N^{2} \Big)^{1+\varepsilon}
\sum_{n \leq N} |a_n|^2,
 \end{equation}
which saves a factor $T_1 T_2$ in the second, ``off-diagonal,'' term compared to \eqref{eq:BlomerBoundLocal}, via an oscillatory integral.  The proof of Theorem \ref{thm:spectralsumlocal} also uses an oscillatory integral for an extra savings, but it is a technical challenge to combine these two sources of savings and convert Theorem \ref{thm:spectralsumlocal} into a dyadic version with a secondary term of the same size.
It should be noted that Blomer's estimate arises by applying absolute values to the $GL_3$ Kloosterman sum, and estimating everything trivially (analogously to applying the Weil bound for Kloosterman sums).
One can view the quality of a large sieve inequality for a family $\mathcal{F}$ as a measure of how well one may average with the family.  As such, it is desirable to have strong results.   

There are also large sieve-type results in higher rank due to Duke and Kowalski \cite{DukeKowalski}, Venkatesh \cite{Venkatesh}, and Blomer-Buttcane-Maga \cite{BlomerButtcaneMaga}, but these study the conductor (or level) aspect.
By adapting the method of \cite[Theorem 4]{DukeKowalski}, 
one could use duality and the convexity bound for Rankin-Selberg $L$-functions on $GL_3 \times GL_3$ to attempt to obtain estimates on the left hand side of \eqref{eq:spectralsumlocal}.  However, this method requires
$N$ to be very large compared to $T_1 + T_2$ to give a strong bound.


The $GL_3$ Bruggeman-Kuznetsov formula relates these spectral sums to a sum of $GL_3$ Kloosterman sums.  The main technical contribution of this paper is to analyze multilinear forms with these Kloosterman sums.  We will be using the Bruggeman-Kuznetsov formula in the form of \cite[Proposition 4]{Blomer}\footnote{A corrected version of the formula can be found 
in \cite[Theorem 6]{BlomerButtcaneMaga}}.  The geometric side of this formula involves the $GL_3$ Kloosterman sums, which we now define.
The (long element) Kloosterman sum is
\begin{multline}
 S(m_1, m_2, n_1, n_2;D_1, D_2) = \mathop{\sum \sum}_{\substack{B_1, C_1 \shortmod{D_1} \\ B_2, C_2 \shortmod{D_2} \\ (B_1, C_1, D_1)  = (B_2, C_2, D_2) =1 \\ D_1 C_2 + B_1 B_2 + C_1 D_2 \equiv 0 \shortmod{D_1 D_2}} } e\Big(\frac{m_1 B_1 + n_1 (Y_1 D_2 - Z_1 B_2)}{D_1} \Big)
 \\
 e\Big(\frac{m_2 B_2 + n_2 (Y_2 D_1 - Z_2 B_1)}{D_2} \Big),
\end{multline}
where $Y_1, Y_2, Z_1, Z_2$ are defined (chosen) so that
\begin{equation}
 Y_1 B_1 + Z_1 C_1 \equiv 1 \pmod{D_1}, \qquad Y_2 B_2 + Z_2 C_2 \equiv 1 \pmod{D_2}.
\end{equation}
Bump, Friedberg, and Goldfeld \cite[Lemmas 4.1 and 4.2]{BFG} have shown that the above sum is well-defined,
meaning that the
value of the sum is independent of the choices of the $Y_i$ and $Z_i$, and
the coset representatives of the $B_i$ and $C_i$.

Define
\begin{equation}
 \mathcal{S} = \mathcal{S}(\alpha, \beta, \gamma) = \sum_{D_1, D_2, m, n}
\gamma_{D_1, D_2} \alpha_m \beta_n S(1,m,n,1,D_1, D_2),
\end{equation}
where $\alpha_m, \beta_n, \gamma_{D_1,D_2}$ are finite sequences.
For our application to the spectral large sieve, we are most interested in the case where $|\gamma_{D_1, D_2}| \leq 1$.   Especially in light of its connections to the large sieve, it is fundamental to estimate $\mathcal{S}$, but it is also of independent interest.
 Our main result is
\begin{mytheo}
\label{thm:BilinearKloosterman}
Suppose that $\alpha_m$, $\beta_n$, and $\gamma_{D_1, D_2}$ are complex sequences supported on $m, n \leq N$, $D_1 \leq X_1$, and $D_2 \leq X_2$.  Furthermore suppose that
$|\gamma_{D_1, D_2}| \leq 1$.  For an arbitrary finitely supported sequence $\beta = (\beta_n)$, let
\begin{equation}
\label{eq:MbetaDef}
M(\beta) =  \sum_{q \leq \min(X_1, X_2)} \sum_{d_1| q} \frac{d_1}{q} \sum_{\substack{c \leq \frac{X_1}{q} \\ (c,q) = 1}} \thinspace \sumstar_{t \shortmod{c}} 
 \Big|  
  \sum_{(n,q) = d_1} \beta_n e\Big(\frac{tn}{c}\Big)
 \Big|^2
\end{equation}
where $\Sigma^*$ denotes that $t$ is restricted by $(t,c) = 1$.
Then
\begin{equation}
\label{eq:BilinearKloosterman}
|\mathcal{S}| \ll (X_1 X_2)^{1+\varepsilon}  M(\alpha)^{1/2} M(\beta)^{1/2}.
\end{equation}
\end{mytheo}
For some special choices of coefficients $\alpha_m, \beta_n$ (e.g. Dirichlet series coefficients of an $L$-function), one could potentially use alternative techniques to handle small $c$, which explains why we have stated Theorem \ref{thm:BilinearKloosterman} in this form.
For arbitrary coefficients, one cannot do better than the 
large sieve inequality (see \cite[Theorem 7.11]{IK}), which implies
\begin{equation}
 M(\beta) \ll (X_1^2 + N ) X_1^{\varepsilon} \| \beta \|^2,
\end{equation}
where here and throughout the paper we use the notation (for an arbitrary sequence $\beta$ of finite support)
\begin{equation}
 \| \beta\| = \Big(\sum_{n \in \mz} |\beta_n|^2 \Big)^{1/2}.
\end{equation}

Hence we immediately derive
\begin{mycoro}
\label{coro:Sbound}
 With the same conditions and notation as Theorem \ref{thm:BilinearKloosterman}, we have
 \begin{equation}
  |\mathcal{S}| \ll (X_1 X_2)^{\varepsilon} (X_1 X_2) (X_1^2 +  N)^{1/2} (X_2^2 + N)^{1/2} \|\alpha\| \thinspace \|\beta \|.
 \end{equation}
\end{mycoro}
For the applications to the $GL_3$ spectral large sieve inequality, the formulation in Theorem \ref{thm:BilinearKloosterman} is better, because one can obtain additional savings using a hybrid large sieve inequality, which includes an archimedean integral.

For some ranges of the parameters, the following result is superior to Theorem \ref{thm:BilinearKloosterman}:
\begin{mytheo}
\label{thm:BilinearKloostermanHversion}
 Let $1 \leq H_1 \leq X_1$, $1 \leq H_2 \leq X_2$.  Then
 \begin{equation}
 \label{eq:SboundWithHParameters}
  |\mathcal{S}| \ll (X_1 H_2 + X_2 H_1) (X_1 X_2)^{\varepsilon} M^*(\alpha)^{1/2} M^*(\beta)^{1/2} + (X_1 X_2)^{3/2+\varepsilon} N^{1+\varepsilon} \| \alpha \| \| \beta\| (H_1^{-1} + H_2^{-1}),
 \end{equation}
 where $M^*(\beta)$ is defined as in \eqref{eq:MbetaDef}, but with $q$ restricted by $q \leq \min(H_1, H_2)$.
\end{mytheo}
Remarks.  In case $H_1 = X_1$, $H_2 = X_2$ the first term in \eqref{eq:SboundWithHParameters} reduces to Theorem \ref{thm:BilinearKloosterman} (and the second term may be dropped).  For the opposite extreme $H_1 = H_2 = 1$, the latter term corresponds to the ``Weil bound'' (see \eqref{eq:SboundTrivial} below) while the first term may be dropped. The restrictions $1 \leq H_i \leq X_i$ may be dropped from the statement of Theorem \ref{thm:BilinearKloostermanHversion}, however then the result is worse than Theorem \ref{thm:BilinearKloosterman} or \eqref{eq:SboundTrivial} below.

Again, the large sieve implies
\begin{mycoro}
 \label{coro:SboundWithHParameters}
  Let $1 \leq H_1 \leq X_1$, $1 \leq H_2 \leq X_2$.  Then
 \begin{multline}
  |\mathcal{S}| \ll \Big[(X_1 H_2 + X_2 H_1)  (X_1^2 + N)^{1/2} (X_2^2 + N)^{1/2}  
  \\
  + \frac{(X_1 X_2)^{3/2} N}{H_1} + \frac{(X_1 X_2)^{3/2} N}{H_2}
  \Big] (X_1 X_2 N)^{\varepsilon} \| \alpha \| \thinspace \| \beta \|.
 \end{multline}
\end{mycoro}
Remarks. For $N$ large, say $N \gg X_1^2 + X_2^2$, Corollary \ref{coro:SboundWithHParameters} is optimized with $H_1 = X_1^{3/4} X_2^{1/4}$, $H_2 = X_1^{1/4} X_2^{3/4}$, and reduces to a bound 
that can be seen to be inferior to Corollary \ref{coro:Sbound}.  On the other hand, if $N \ll \min(X_1^2, X_2^2)$, then the optimal bound occurs with $H_1 = N^{1/2} X_1^{1/4} X_2^{-1/4}$, $H_2 = N^{1/2} X_2^{1/4} X_1^{-1/4}$, and gives
\begin{equation}
 |\mathcal{S}| \ll (X_1 X_2)^{5/4} N^{1/2} (X_1^{1/2} + X_2^{1/2}) (X_1 X_2 N)^{\varepsilon} \| \alpha \| \thinspace \| \beta \|.
\end{equation}

Recently, Buttcane \cite{Buttcane1} \cite{Buttcane2} has developed Mellin-Barnes integral representations for the weight functions occuring on the Kloosterman sum side of the Bruggeman-Kuznetsov formula.  Blomer and Buttcane \cite{BlomerButtcane} have used this formulation, with additional ideas, to obtain a subconvexity result for $GL_3$ Maass forms in the spectral aspect.  It could be interesting to investigate if these alternative integral representations  lead to additional savings in the spectral large sieve.  Our preliminary calculations indicate this could be rather complicated, and since our main focus here is on the arithmetical aspects of the problem (rather than the archimedean integrals), we leave this for another occasion.

\section{Acknowledgments}
I thank Valentin Blomer, Jack Buttcane, and the referee for numerous suggestions and corrections that improved the quality of the paper.

\section{Heuristic remarks}
\subsection{On theorem \ref{thm:BilinearKloosterman}}
\label{section:HeuristicBilinearKloosterman}
We include a few remarks of an informal nature indicating that Theorem \ref{thm:BilinearKloosterman} is in a somewhat robust form, at least, under the assumption that $X_1$ and $X_2$ are not highly asymmetrical in size.

The Weil-type bound of Steven \cite{Stevens} (see \cite[Lemma 3]{Blomer}) implies
\begin{equation}
 \sum_{D_1 \leq X_1} \sum_{D_2 \leq X_2} |S(1,m,n,1, D_1, D_2)| \ll (X_1 X_2)^{3/2 + \varepsilon} (mn)^{\varepsilon},
\end{equation}
and therefore the trivial bound applied to $\mathcal{S}$ along with Cauchy's inequality gives
\begin{equation}
\label{eq:SboundTrivial}
 |\mathcal{S}| \ll (X_1 X_2)^{3/2 + \varepsilon} N^{1+\varepsilon} \| \alpha \| \| \beta \|.
\end{equation}
Therefore, for large $N$, Corollary \ref{coro:Sbound} saves an additional factor $(X_1 X_2)^{1/2}$ over \eqref{eq:SboundTrivial}.

In case $(D_1, D_2) =
1$, then from \cite[Property 4.9]{BFG}, we have
\begin{equation}
\label{eq:KloostermanFactorizationCoprimeModuli}
S(m_1, m_2, n_1, n_2, D_1, D_2) = S(D_2 m_1, n_1, D_1) S(D_1 m_2, n_2,
D_2) 
\end{equation}
so the contribution to $\mathcal{S}$ from $(D_1,
D_2) = 1$, say
$\mathcal{S}'$, is
\begin{equation}
\mathcal{S}' = \sum_{\substack{m, n \\ (D_1, D_2) = 1 }} \gamma_{D_1, D_2}   \alpha_m \beta_n 
 S(D_2 , n, D_1)
 S(D_1 , m, D_2).
\end{equation}
It could so happen that $\gamma_{D_1, D_2} $ always has the same sign as $ \sum_{m,n} \alpha_m \beta_n S(D_2 , n, D_1) S(D_1 , m, D_2)$, so it should be essentially impossible
to do better than bounding $\mathcal{S}'$ as follows:
\begin{equation}
|\mathcal{S}'| \leq \sum_{(D_1, D_2)=1} \Big|\sum_n \beta_n S(D_2, n, D_1) \sum_{m} \alpha_m S(D_1  , m, D_2)   \Big|.
\end{equation}
By an application of Cauchy's inequality, we have
\begin{equation}
\label{eq:S'inequality}
 |\mathcal{S'}| \leq \Big(\sum_{(D_1, D_2)=1} \Big|\sum_n \beta_n S(D_2, n, D_1) \Big|^2 \Big)^{1/2} (\dots)^{1/2},
\end{equation}
with the dots representing a similar term.  
Next we drop the condition $(D_2, D_1) = 1$ and extend the sum over $D_2$ to $D_2 \leq M D_1$ where $M$ is the unique integer satisfying $X_2 \leq M D_1 < X_2 + D_1$ (this extension is presumably rather wasteful in case $X_2$ is much smaller than $X_1$).
Then we have
\begin{equation}
\label{eq:KloostermanProductCompleteSum}
 \sum_{D_2 \leq M D_1}  S(n_1, D_2, D_1) S(n_2, D_2, D_1) = M D_1 \sumstar_{x \shortmod{D_1}} e\Big(\frac{x(n_1 -n_2)}{D_1}\Big), 
\end{equation}
so the first expression in parentheses on the right hand side of \eqref{eq:S'inequality} satisfies
\begin{equation}
 ( \dots) \leq (X_1 + X_2)  \sum_{D_1 \leq X_1} \thinspace \sumstar_{x \shortmod{D_1}} \Big| \sum_{n} \beta_n e\Big(\frac{xn}{D_1}\Big) \Big|^2.
\end{equation}
A similar bound holds for the second factor in \eqref{eq:S'inequality}, of course.
Therefore, by the large sieve inequality, we have
\begin{equation}
\label{eq:S'bound}
 |\mathcal{S}'| \leq (X_1 + X_2) (X_1^2 + N)^{1/2} (X_2^2 + N)^{1/2} \| \alpha \| \| \beta\|.
\end{equation}
This gives a limitation to the final estimates we wish to obtain for $\mathcal{S}$.  One observes that the bound \eqref{eq:S'bound} is superior to that of Corollary \ref{coro:Sbound} by a factor $\min(X_1, X_2)$, which arises in the proof from considering $D_1$ and $D_2$ with a common factor.

The opposite extreme of $(D_1, D_2) = 1$ is
$D_1 = D_2$.  For simplicity consider $D_1 = D_2 = p$, prime.  
In this case, we have (see \cite[Property 4.10]{BFG} or Lemma \ref{lemma:Sevaluation} below)
\begin{equation}
\label{eq:KloostermanEvaluationPrimePrime}
 S(1,m,n,1,p,p) = S(m,0;p) S(n,0;p) + p.
\end{equation}
Therefore, if $p | (m,n)$ the Kloosterman sum is of order $p^2$, while if $p \nmid m$, $p \nmid n$, it is of order $p$.  The term $p$ gives the dominant contribution, because in the situation when the Kloosterman sum has order $p^2$ (i.e., $p|(m,n)$), the rarity in $m$ and $n$ has relatively frequency $p^{-2}$, which is a net saving by a factor $p$.  These terms give to $\mathcal{S}$ an amount, say $\mathcal{S}''$, given by
\begin{equation}
 \mathcal{S}'' = \sum_{p \leq \min(X_1, X_2)} (p + 1) \gamma_{p,p} \sum_{(m,p) = 1} \alpha_m \sum_{(n,p) = 1} \beta_n.
\end{equation}
If say $X_1 = X_2 = X$, then
\begin{equation}
\mathcal{S}'' \ll X^2 \Big|\sum_m \alpha_m \Big| \cdot \Big|\sum_n \beta_n \Big|,
\end{equation}
which is best-possible since the sum of $(p+1) \gamma_{p,p}$ may have the same sign as $\sum_m \alpha_m \sum_n \beta_n$.
A bound of this magnitude is included with $c = q = d_1 = 1$ in \eqref{eq:MbetaDef} and \eqref{eq:BilinearKloosterman}.  Cauchy's inequality applied to $\mathcal{S}''$ gives
\begin{equation}
 \mathcal{S}''
 \ll X^2 N \| \alpha \| \| \beta \|.
\end{equation}
This matches the bound in Corollary \ref{coro:Sbound} for $N$ large and $X_1 = X_2$.

Of course, in actual practice it is necessary to treat all possible values of $\gcd(D_1, D_2)$ that ``interpolate'' the two extremes $(D_1, D_2) = 1$, and $D_1 = D_2$, and indeed this is accomplished in the proof of Theorem \ref{thm:BilinearKloosterman}.  In fact, this is the main difficulty in the proof.

The above remarks indicate that the quality of Theorem \ref{thm:BilinearKloosterman} comes largely from terms where $(D_1, D_2)$ is large.  This might be surprising in light of the relative rarity of such terms.

\subsection{The $GL_2$ spectral large sieve}
The spectral large sieve for $SL_2(\mz) \backslash \mh$ was originally proved by Iwaniec \cite{IwaniecLargeSieve}, while the case of congruence subgroups was extensively developed by Deshouillers and Iwaniec \cite{DeshouillersIwaniec}.
Here we sketch a proof inspired by Jutila \cite[Section 3]{Jutila}, since
we shall use this method as a motivating guide for the more challenging $GL_3$ case.  Recall that the $GL_2$ spectral large sieve states
\begin{equation}
\sum_{T \leq t_j \leq T + \Delta} \frac{1}{R_j}  \Big| \sum_{n \leq N} a_n \lambda_j(n) \Big|^2 \ll (NT)^{\varepsilon} (\Delta T +N) \sum_n |a_n|^2,
\end{equation}
where $R_j$ is given by \eqref{eq:Rjdef} (but for $\phi_j$ a Hecke-Maass cusp form on $SL_2(\mz)$), and $1 \leq \Delta \leq T$.

The $GL_2$ Bruggeman-Kuznetsov formula gives
\begin{equation}
 \sum_{T \leq t_j \leq T + \Delta} \frac{1}{R_j} \Big| \sum_{N/2 < n \leq N} a_n \lambda_j(n) \Big|^2 \ll \Delta T \sum_n |a_n|^2 + 
 \mathcal{K},
\end{equation}
where
\begin{equation}
 \mathcal{K} = \sum_{N/2 < m,n \leq N} a_m \overline{a_n} \sum_{c=1}^{\infty} \frac{S(m,n;c)}{c} B\Big(\frac{ \sqrt{mn}}{c}\Big),
\end{equation}
and where $B(x)$ is a certain integral transform of a nonnegative weight function $h$ that is $\gg 1$ for $T \leq t \leq T + \Delta$.  For an appropriate smooth choice of $h$, $B(x)$ is very small unless $x \gg \Delta T^{1-\varepsilon}$.  Then by a Mellin transform, we have approximately that for $x \asymp X$, 
\begin{equation}
 B(x) \approx \Delta T \int_{|t| \ll X} X^{-1} x^{it} b(t) dt,
\end{equation}
where $b(t) \ll 1$.  Here $b$ depends on $X$, but not on $x$.
Applying this formula to $\mathcal{K}$, we derive
\begin{equation}
\mathcal{K} \lessapprox \sum_{C \text{ dyadic} } \frac{\Delta T}{CX} 
 \int_{|t| \ll X} b(2t)  \sum_{c \asymp C} c^{-2it} \sumstar_{a \shortmod{c}}
 \Big(\sum_{m \asymp N} a_m m^{it} e\Big(\frac{am }{c}\Big) \Big)
 \Big(\sum_{n \asymp N} \overline{a_n} n^{it} e\Big(\frac{\overline{a} n }{c}\Big) \Big)dt.
\end{equation}
The hybrid large sieve inequality of Gallagher \cite{Gallagher} states
\begin{equation}
 \int_{|t| \leq X}  \sum_{c \leq C} \thinspace \sumstar_{a \shortmod{c}} \Big| \sum_{n \leq N} a_n n^{-it} e\Big(\frac{an}{c}\Big) \Big|^2 dt \ll (C^2 X + N) \sum_n |a_n|^2.
\end{equation}
Applying this to $\mathcal{K}$ after a use of Cauchy-Schwarz, and using $X \asymp \frac{N}{C}$, and $C \ll \frac{N}{\Delta T} (NT)^{\varepsilon}$, we derive
\begin{equation}
\mathcal{K} \lessapprox \Delta T \sum_{C \text{ dyadic} } (CX)^{-1} (C^2 X + N) \sum_n |a_n|^2 \ll N (NT)^{\varepsilon} \sum_n |a_n|^2.
\end{equation}

The main observation is that the $GL_1$ hybrid large sieve inequality drives the final estimations, and only rather crude information is required on $B$, namely its truncation and size of its Mellin transform.  The hybrid aspect of the large sieve is able to recover the loss in separation of variables in $B$.  

For later use, we shall require a different version (though morally equivalent) of the hybrid large sieve than that given by Gallagher.  The following is a special case of \cite[Lemma 6.1]{Young}.
\begin{mylemma}
\label{lemma:variantsieve}
Let $b_m$ be  arbitrary complex numbers, and suppose $Y \gg 1$.  Then
\begin{equation}
\label{eq:9.2}
 \int_{1}^{2} \sum_{b \leq B} \thinspace \sumstar_{x \shortmod{b} } \Big| \sum_{N \leq m < N + M} b_m \e{xm}{b} e\Big(\frac{tm}{Y}\Big) \Big|^2 dt \ll (B^2  + Y) \sum_{N \leq m < N+ M} |b_m|^2.
\end{equation}
\end{mylemma}

\section{Preliminary arithmetical results}
For ease of reference, we collect here some results.
First we need an individual ``Weil-type" bound.  This estimate was proved by Stevens \cite{Stevens} but without explicit dependence on the $m_i$ and $n_i$, which was subsequently investigated by Buttcane \cite[Theorem 4]{Buttcane1}.
\begin{mylemma}
\label{lemma:KloostermanUpperBound}
For $m_1, m_2, n_1, n_2 \in \mz \setminus \{0 \}$, we have
\begin{equation}
S(m_1, m_2, n_1, n_2, D_1, D_2) \ll (D_1 D_2)^{1/2 + \varepsilon} ((D_1, D_2) (m_1 n_2, [D_1, D_2]) (m_2 n_1, [D_1, D_2]))^{1/2}.
\end{equation}
\end{mylemma}
This estimate is not sharp for $(D_1, D_2) >1$, but it is difficult to extract clean results from the literature (see \cite[Theorem 3.7]{DabrowskiFisher}).  We may obtain some easy improvements by way of explicit computations in some important special cases:
\begin{mylemma}[\cite{BFG}]
\label{lemma:Sevaluation}
Suppose $l \geq 1$.  Then
\begin{equation}
S(m_1, m_2, n_1, n_2, p, p^l) = S(n_1, 0;p) S(m_2, n_2 p, p^l) + S(m_1, 0;p) S(n_2, m_2 p ; p^l) + \delta_{l=1} (p-1).
\end{equation}
\end{mylemma}

\begin{mylemma}
\label{lemma:KloostermanVanishes}
Suppose $b \geq 1$, $c \geq 2$, and $
(\alpha \beta , p) = 1$.  Then
\begin{equation}
\label{eq:KloostermanVanishes}
S(\alpha,  \beta p^b, p^c) = 0.
\end{equation}
\end{mylemma}
\begin{proof}
If $b \geq c$, then $S(\alpha, \beta p^b, p^c) = S(1,
0, p^c) = 0$ since $c \geq 2$,  so suppose $c \geq b+1$.
Opening the Kloosterman sum as a sum over $x
\pmod{p^c}$, we change variables $x = x_1 (1 + p^{c-b} x_2)$, where $x_1$ runs
modulo $p^{c-b}$ (coprime to $p$) and $x_2$ runs modulo $p^{b}$.  Then
$\overline{x} \equiv \overline{x_1} \pmod{p^{c-b}}$, and so
\begin{equation}
 S(\alpha, \beta p^b, p^c) = \sumstar_{x_1 \shortmod{p^{c-b}}} \sum_{x_2
\shortmod{p^b}} e\Big(\frac{\alpha x_1 (1 + p^{c-b} x_2) + \beta p^b
\overline{x_1}}{p^c} \Big).
\end{equation}
The sum over $x_2$ then vanishes since $b \geq 1$ and $(\alpha x_1, p) = 1$.
\end{proof}

Consider $S(a,y,x,b,D_1, D_2)$ with $(a,D_1) = (b, D_2) = 1$.  Then define its 
(partial, middle two-variable) 
Fourier transform by
\begin{equation}
 \widehat{S}(a, u, t, b, D_1, D_2) = \frac{1}{D_1 D_2} \sum_{x \shortmod{D_1}}
 \sum_{y \shortmod{D_2}}
S(a, y, x, b, D_1, D_2) e\Big(\frac{-x t}{D_1}\Big) e\Big(\frac{-y u}{D_2}\Big),
\end{equation}
so that the Fourier inversion formula reads
\begin{equation}
 S(a,m, n, b, D_1, D_2) =
 \sum_{t \shortmod{D_1}} 
  \sum_{u \shortmod{D_2}}
 e\Big(\frac{tn}{D_1} + \frac{u m}{D_2}
\Big) \widehat{S}(a, u, t, b, D_1, D_2).
\end{equation}
Define
\begin{equation}
\label{eq:Rdefinition}
\mathcal{R}(t, D_1, D_2) = \max_{(ab, D_1) = 1} \sum_{u \shortmod{D_2}} |\widehat{S}(a,u,bt, 1, D_1, D_2)|. 
\end{equation}
Remark.  Using elementary properties of the Kloosterman sums, we may alternatively use the definition
\begin{equation}
\label{eq:RdefinitionAlternate}
\mathcal{R}(t, D_1, D_2) = \max_{(b, D_1) = 1} \sum_{u \shortmod{D_2}} |\widehat{S}(1,u,bt, 1, D_1, D_2)|. 
\end{equation}
This follows by using that $S(a,y,x,1, D_1, D_2) = S(1, y, a x, 1, D_1, D_2)$ (see \cite[Property 4.3]{BFG}), so that after a change of variables we derive 
\begin{equation}
\label{eq:ShatFormulaMovingaToOtherSide}
 \widehat{S}(a,u,bt,1, D_1, D_2) = \widehat{S}(1,u, \overline{a} b t, 1, D_1, D_2).
\end{equation}

The presence of the maximum in \eqref{eq:Rdefinition} is to facilitate the use of the Chinese Remainder Theorem which leads to a more pleasant multiplicative structure for $\mathcal{R}$:
\begin{mylemma}
 The function $\mathcal{R}(t, D_1, D_2)$ is jointly multiplicative in $t, D_1, D_2$.
\end{mylemma}
 \begin{proof}
Say $D_1 = C_1 E_1$ and $D_2 = C_2 E_2$ with $(C_1 C_2, E_1 E_2) = 1$.  Also
write $x = x_C E_1 \overline{E_1} + x_E C_1 \overline{C_1}$, and similarly $y =
y_C E_1 \overline{E_1} + y_E C_1 \overline{C_1}$, where $x_C, y_C, x_E, y_E$ run modulo $C_1, C_2, E_1, E_2$, respectively.
Then using \cite[Property 4.15]{BFG}, we
have
\begin{multline}
 \widehat{S}(a, u, t, 1, D_1, D_2) =  \sum_{x_C, y_C } S(\overline{E_1}^2 E_2 a, \overline{E_2}^2 E_1 y_C, x_C, 1, C_1, C_2)
 e\Big(\frac{-x_C \overline{E_1} t}{C_1}\Big) e\Big(\frac{- y_C \overline{E_2} u}{C_2}\Big)
 \\
\frac{1}{D_1 D_2} \sum_{x_E,
y_E}
 S(\overline{C_1}^2 C_2 a, \overline{C_2}^2 C_1 y_E, x_E, 1, E_1, E_2)
 e\Big(\frac{-x_E \overline{C_1} t}{E_1}\Big) e\Big(\frac{- y_E \overline{C_2} u}{E_2}\Big).
\end{multline}
Changing variables $y_C \rightarrow E_2^2 \overline{E_1} y_C$, $y_E \rightarrow
C_2^2 \overline{C_1} y_E$, we derive
\begin{multline}
 \widehat{S}(a, u, t, 1, D_1, D_2) = \frac{1}{C_1 C_2} \sum_{x_C, y_C} S( \overline{E_1}^2 E_2 a,
y_C, x_C, 1, C_1, C_2)
 e\Big(\frac{-x_C \overline{E_1} t}{C_1}\Big) e\Big(\frac{- y_C E_2 \overline{E_1} u}{C_2}\Big)
 \\
 \frac{1}{E_1 E_2} \sum_{x_E, y_E}
 S(\overline{C_1}^2 C_2 a, y_E, x_E, 1, E_1, E_2)
 e\Big(\frac{-x_E \overline{C_1} t}{E_1}\Big) e\Big(\frac{- y_E C_2 \overline{C_1} u}{E_2}\Big).
\end{multline}
Therefore, 
\begin{equation}
\label{eq:ShatFactorizationFormula}
 \widehat{S}(a, u, t, 1, C_1 E_1, C_2 E_2) = \widehat{S}(\overline{E_1}^2 E_2 a, u E_2 \overline{E_1}, t \overline{E_1}, 1, C_1, C_2)
\widehat{S}(\overline{C_1}^2 C_2 a, u C_2 \overline{C_1}, t \overline{C_1}, 1, E_1, E_2).
 \end{equation}

Using \eqref{eq:ShatFactorizationFormula},
we
 derive that
\begin{multline}
\label{eq:RfactorizationFormulaInProof}
\mathcal{R}(t, C_1 E_1, C_2 E_2) = 
 \max_{(ab, C_1 E_1) = 1} \sum_{u \shortmod{C_2 E_2}} 
 |\widehat{S}(\overline{E_1}^2 E_2 a, u E_2 \overline{E_1}, bt \overline{E_1}, 1, C_1, C_2)|
\\
|\widehat{S}(\overline{C_1}^2 C_2 a, u C_2 \overline{C_1}, bt \overline{C_1}, 1, E_1, E_2)|.
\end{multline}
In the right hand side of \eqref{eq:RfactorizationFormulaInProof}, the first line only depends on $u$ modulo $C_2$, and $a,b$ modulo $C_1$, while the second line only depends on $u$ modulo $E_2$, and $a,b$ modulo $E_1$.  Therefore, we have
\begin{equation}
\begin{split}
\mathcal{R}(t, C_1 E_1, C_2 E_2) = 
 \max_{(a_C b_C, C_1) = 1} \sum_{u_C \shortmod{C_2}} 
 |\widehat{S}(\overline{E_1}^2 E_2 a_C, u_C E_2 \overline{E_1}, b_C t \overline{E_1}, 1, C_1, C_2)|
\\
 \max_{(a_E b_E, E_1) = 1} \sum_{u_E \shortmod{E_2}} 
|\widehat{S}(\overline{C_1}^2 C_2 a_E, u_E C_2 \overline{C_1}, b_E t \overline{C_1}, 1, E_1, E_2)|.
\end{split}
\end{equation}
Changing variables $a_C \rightarrow E_1^2 \overline{E_2} a_C$, $u_C \rightarrow E_1 \overline{E_2} u_C $, $b_C \rightarrow E_1 b_C$, and similarly for $a_E$, $u_E$, and $b_E$, we derive $\mathcal{R}(t, C_1 E_1, C_2 E_2) = \mathcal{R}(t, C_1, C_2) \mathcal{R}(t, E_1, E_2)$, as desired.
%
\end{proof}
\begin{mylemma}
\label{lemma:R'property}
 Let 
 \begin{equation}
\mathcal{R}'(u, D_1, D_2) = \max_{\substack{(a, D_1) = 1 \\ (b, D_2) = 1}} \sum_{t \shortmod{D_1}} |\widehat{S}(a,bu,t, 1, D_1, D_2)|. 
\end{equation}
Then
\begin{equation}
 \mathcal{R'}(u, D_1, D_2) = \mathcal{R}(u, D_2, D_1).
\end{equation}
\end{mylemma}
\begin{proof}
 Using $S(a,y,x,1,D_1, D_2) = S(x,1,a,y,D_1, D_2) = S(1,x,y,a, D_2, D_1)$ (see Properties 4.5 and 4.4 of \cite{BFG}), along with $S(1,x,y,a, D_2, D_1) = S(1, a x, y, 1, D_2, D_1)$ (\cite[Property 4.3]{BFG}), we derive that
 \begin{multline}
  \widehat{S}(a,bu,t,1,D_1, D_2) = \frac{1}{D_1 D_2} \sum_{x \shortmod{D_1}}
 \sum_{y \shortmod{D_2}}
S(1,  x, y, 1, D_2, D_1) e\Big(\frac{-x \overline{a} t}{D_1}\Big) e\Big(\frac{-y bu}{D_2}\Big)
\\
= \widehat{S}(1, \overline{a}t, bu, 1, D_2, D_1).
 \end{multline}
From this, and using \eqref{eq:RdefinitionAlternate}, we complete the proof.
\end{proof}

\begin{mylemma}
\label{lemma:ShatReverseModuli}
 We have
 \begin{equation}
  \widehat{S}(a,u,t,b,D_1, D_2) = \widehat{S}(b,t,u,a, D_2, D_1).
 \end{equation}
\end{mylemma}
\begin{proof}
 A minor variation of the proof of Lemma \ref{lemma:R'property} gives the result. 
\end{proof}

\begin{mydefi}[Definition of $\nu$]
 Suppose $p$ is a prime.  If $n \in \mz$, we define $\nu_p(n)$ to be the standard $p$-adic valuation of $n$. 
 If $k \geq 1$ and $t \in \mz/p^k \mz$ we define $\nu_p(t)$ to be the largest $j \leq k$ such that $t \equiv 0 \pmod{p^j}$.  
\end{mydefi}
Remark.  One may easily check that $\nu_p(t)$ is well-defined for $t \in \mz/p^k \mz$; without the restriction $j \leq k$, two coset representatives may have different $p$-adic valuations.

\begin{mylemma}
\label{lemma:Rbound}
Suppose $(ab,p) = 1$, and set $\nu = \nu_p(t)$.  Then
\begin{equation}
 \mathcal{R}(t, p^k, p^l) 
 \leq (k+1)p^{l} + p^{\nu + l}  \delta(\nu \leq \tfrac23 \min(k,l) ).
\end{equation}
\end{mylemma}
Remark. For $k \neq l$, our proof shows that we can replace $p^{\nu+l}$ by $p^{\frac{\nu}{2} + l}$, and restrict $\nu \leq \tfrac12 \min(k,l)$.  It is plausible one can save this factor $p^{\nu/2}$ for $k=l$, 
but since this would not improve Theorem \ref{thm:BilinearKloosterman}, and since our proof is already quite long, we avoided this line of inquiry.  The key point in Lemma \ref{lemma:Rbound} is that the ``loss'' from the factor $p^{\nu}$ is countered by the condition $p^{\nu} | t$.  This has the practical effect that large values of $\nu$ give essentially the same bound as for $\nu = 0$.

The proof of Lemma \ref{lemma:Rbound} is given in Section \ref{section:Rbound}.

\begin{mycoro}
\label{coro:Rbound}
 We have
\begin{equation}
 \mathcal{R}(t,D_1, D_2) \ll D_2 (D_1 D_2)^{\varepsilon}  \sum_{\substack{d|t \\ d^3 | (D_1, D_2)^2 }} d.
 \end{equation}
\end{mycoro}
\begin{proof}
Since both sides are multiplicative, it suffices to check on prime powers, in which case it follows immediately from Lemma \ref{lemma:Rbound}.
\end{proof}

\begin{mylemma}
\label{lemma:divisortypebound}
 Let $q \leq X$.  The number of integers $n \leq X$ that share the same set of prime divisors as $q$ (that is, such that $\nu_p(n) \geq 1$ iff $\nu_p(q) \geq 1$ for all primes $p$) is $\ll_{\varepsilon} X^{\varepsilon}$, for any $\varepsilon > 0$.
\end{mylemma}
\begin{proof}
 This is similar to a divisor-type bound.  Suppose that the prime factors occuring in $q$ are $p_1, \dots, p_r$.  Then by Rankin's trick, we have
 \begin{equation}
 \label{eq:divisortypeboundestimateinproof}
  \sum_{\substack{n = p_1^{a_1} \dots p_r^{a_r} \leq X \\ a_i \geq 1, \text{ all } i }} 1 \leq  \sum_{a_1=1}^{\infty} \dots \sum_{a_r=1}^{\infty} \Big(\frac{ X}{p_1^{a_1} \dots p_r^{a_r}} \Big)^{\varepsilon} = \frac{X^{\varepsilon}}{(p_1^{\varepsilon}-1) \dots (p_r^{\varepsilon}-1)}.
 \end{equation}
Given $\varepsilon > 0$, there are finitely many primes such that $p^{\varepsilon} \leq 2$.  Then with $C(\varepsilon) = \prod_{p: p^{\varepsilon} \leq 2} (p^{\varepsilon}-1)^{-1}$, we may bound the right hand side of \eqref{eq:divisortypeboundestimateinproof} by $C(\varepsilon) X^{\varepsilon}$.
\end{proof}

\section{Proof of Theorems \ref{thm:BilinearKloosterman} and \ref{thm:BilinearKloostermanHversion}}
\label{section:bilinearwithKloosterman}
\subsection{Initial decomposition}
Our first steps involve factoring $D_1$ and $D_2$ in appropriate ways and using the Chinese remainder theorem to correspondingly factor the Kloosterman sum.

First we extract the largest divisors of $D_1$ and $D_2$ that are coprime to each other.  Precisely,
write $D_1 = g_1 E_1$, $D_2 = g_2 E_2$, where $(E_1 E_2, g_1 g_2) = 1$, $(E_1, E_2) = 1$, and $g_1$ and $g_2$ have the same set of prime divisors (meaning, $\nu_p(g_1) \geq 1$ iff $\nu_p(g_2) \geq 1$).  
Then by \cite[Property 4.7]{BFG}, we have
\begin{equation}
 S(1,m,n,1,g_1 E_1, g_2 E_2) = S(\overline{g_1}^2 g_2, \overline{g_2}^2 g_1 m, n, 1, E_1, E_2) S(\overline{E_1}^2 E_2, \overline{E_2}^2 E_1 m, n, 1, g_1, g_2).
\end{equation}
By \eqref{eq:KloostermanFactorizationCoprimeModuli}, we have
\begin{equation}
 S(\overline{g_1}^2 g_2, \overline{g_2}^2 g_1 m, n, 1, E_1, E_2) = S(E_2 \overline{g_1}^2 g_2, n, E_1) S(E_1 \overline{g_2}^2 g_1, m, E_2).
\end{equation}

Therefore,
\begin{equation}
 |\mathcal{S}| \leq \sumprime_{g_1, g_2, E_1, E_2} \Big| \sum_{m,n} \alpha_m \beta_n
 S(E_2 \overline{g_1}^2 g_2, n, E_1) S(E_1 \overline{g_2}^2 g_1, m, E_2)
 S(\overline{E_1}^2 E_2, \overline{E_2}^2 E_1 m, n, 1, g_1, g_2)
 \Big|,
\end{equation}
where the prime represents the conditions:
\begin{equation}
 g_1 E_1 \leq X_1, \quad g_2 E_2 \leq X_2, \quad (E_1 E_2, g_1 g_2) = 1, \quad (E_1, E_2) = 1, \quad \nu_p(g_1) \geq 1 \text{ iff } \nu_p(g_2) \geq 1.
\end{equation}

We factor the moduli further by extracting the 
prime factors of $g_1$ and $g_2$ such that $\nu_p(g_1) = \nu_p(g_2)=1$.  
Precisely, write $g_1 = q h_1$, $g_2 = q h_2$ where $q$ is the product of primes such that $\nu_p(g_1) = \nu_p(g_2) = 1$, so that for all $p | h_1 h_2$, $\nu_p(h_1) \geq 2$ or $\nu_p(h_2) \geq 2$, and $(q, h_1 h_2) = 1$.
Then we have
\begin{multline}
 S(\overline{E_1}^2 E_2, \overline{E_2}^2 E_1 m, n, 1, q h_1, q h_2)
 \\
 =
 S(\overline{(h_1 E_1)}^2 h_2 E_2, \overline{(h_2 E_2)}^2 h_1 E_1 m, n, 1, q, q)
 S(\overline{q} \overline{E_1}^2 E_2, \overline{q} \overline{E_2}^2 E_1 m, n, 1,  h_1, h_2).
\end{multline}
By \eqref{eq:KloostermanEvaluationPrimePrime}, and using $(ab,p) = 1$, we have 
\begin{equation}
 S(a, b m, n, 1, p, p) = S(m,0,p) S(n,0,p) + p = \begin{cases}
                                                  p^2-p+1, \quad p |(m,n), \\
                                                  p+1, \quad p \nmid m, p \nmid n,
                                                  \\
                                                  1, \quad p|m, p \nmid n
                                                  \\
                                                  1, \quad p|n, p \nmid m,
                                                 \end{cases}
\end{equation}
and so by the Chinese remainder theorem, if $q$ is squarefree and $(ab,q) = 1$, then 
\begin{equation}
 S(a, b m, n, 1, q, q) = \prod_{\substack{ p | q \\ p \nmid m, p \nmid n}} (p+1) 
 \prod_{\substack{ p | (m,n,q) }} (p^2 -p+1).
\end{equation}

Set $d_1 = (n,q)$ and $d_2 = (m,q)$, and define
\begin{equation}
 A(d_1, d_2,q) =  
 \prod_{\substack{ p | d_1, d_2 }} (p^2 -p+1)
 \prod_{\substack{ p | q, p \nmid d_1, p \nmid d_2}} (p+1).
\end{equation}
Then the above calculations show
\begin{equation}
 S(\overline{(h_1 E_1)}^2 h_2 E_2, \overline{(h_2 E_2)}^2 h_1 E_1 m, n, 1, q, q) = A(d_1, d_2, q).
\end{equation}
One easily checks
\begin{equation}
\label{eq:Abound}
 A(d_1, d_2, q) \ll q^{1+\varepsilon} \frac{(d_1, d_2)^3}{d_1 d_2}.
 \end{equation}

Summarizing the above discussion, we have shown
\begin{multline}
\label{eq:SboundSeparatedSomewhat}
  |\mathcal{S}| \leq \sumprime_{h_1, h_2, q, E_1, E_2}  \sum_{d_1, d_2 | q} A(d_1, d_2, q)  \Big| \sum_{(n,q) = d_1} \sum_{(m,q) = d_2} \alpha_m \beta_n
 S(\overline{q} \overline{E_1}^2 E_2, \overline{q} \overline{E_2}^2 E_1 m, n, 1,  h_1, h_2)
 \\
 S(E_2 \overline{q} \overline{h_1}^2 h_2, n, E_1) S(E_1 \overline{q} \overline{h_2}^2 h_1, m, E_2)
 \Big|,
\end{multline}
where the prime on the sum is updated to represent the conditions:
\begin{gather}
\begin{split}
\label{eq:primesumconditions2}
 q h_1 E_1 \leq X_1, \quad q h_2 E_2 \leq X_2, \quad (E_1 E_2, q h_1 h_2) = 1, \quad (E_1, E_2) = 1, \quad \nu_p(q) \in \{0,1 \},
 \\
(q, h_1 h_2) = 1, \quad  \nu_p(h_1) = 0 \text{ iff } \nu_p(h_2) = 0, \quad p | h_1 h_2 \Rightarrow \nu_p(h_1) \geq 2 \text{ or } \nu_p(h_2) \geq 2.
\end{split}
\end{gather}

Remark. 
Heuristically, the sum over $h_1$ and $h_2$ is somewhat small since both integers share the same prime divisors, and for each prime $p | h_1 h_2$, $p^2$ divides at least one of $h_1, h_2$.  If we let $\mathcal{S}'''$ denote the terms on the right hand side of \eqref{eq:SboundSeparatedSomewhat} with $h_1 = h_2 = 1$, then following the arguments of Section \ref{section:HeuristicBilinearKloosterman}, one can derive
\begin{equation}
 \mathcal{S}''' \ll (X_1 + X_2)^{1+\varepsilon} (X_1^2 + \min(X_1, X_2) N)^{1/2} (X_2^2 + \min(X_1, X_2) N)^{1/2} \|\alpha \| \| \beta \|.
\end{equation}
This is better than our final bound on $\mathcal{S}$ given by Corollary \ref{coro:Sbound} for large $X_1, X_2$, so perhaps a more careful analysis of $h_1$ and $h_2$ could lead to a modest improvement.

If either $h_1$ or $h_2$ is large, then it can be beneficial to estimate the sum with absolute values, exploiting the reduced number of moduli under consideration.  Define $\mathcal{S}_{q h_i \leq H_i}$ to be the sum on the right hand side of 
\eqref{eq:SboundSeparatedSomewhat} with $q h_1 \leq H_1$ and $q h_2 \leq H_2$, and similarly define $\mathcal{S}_{qh_1 > H_1}$ and $\mathcal{S}_{qh_2 > H_2}$ corresponding to the terms with $qh_1 > H_1$ and $qh_2 > H_2$, respectively.
Then we have the decomposition $|\mathcal{S}| \leq \mathcal{S}_{q h_1 > H_1} + \mathcal{S}_{qh_2 > H_2} + \mathcal{S}_{q h_i \leq H_i}$.  In the proof of Theorem \ref{thm:BilinearKloosterman}, we may set $H_1 = X_1$, $H_2 = X_2$, and then $\mathcal{S}_{qh_i > H_i} = 0$, for $i=1,2$, so these terms may be discarded.

\subsection{Large $h_i$}
In this subsection we estimate $\mathcal{S}_{q h_1 > H_1}$ and $\mathcal{S}_{qh_2 > H_2}$.
\begin{mylemma}
\label{lemma:LargeH1Bound}
 We have
 \begin{equation}
 \label{eq:LargeH1Bound1}
  \mathcal{S}_{qh_1 > H_1} \ll H_1^{-1} (X_1 X_2)^{3/2 + \varepsilon} N \| \alpha \| \| \beta \|,
 \end{equation}
and
 \begin{equation}
 \label{eq:LargeH1Bound2}
  \mathcal{S}_{qh_2 > H_2} \ll H_2^{-1} (X_1 X_2)^{3/2 + \varepsilon} N \| \alpha \| \| \beta \|.
 \end{equation}
\end{mylemma}
\begin{proof}
 Define
\begin{equation}
T(m,n,D_1, D_2) = \max_{\substack{(a, D_1) = 1 \\ (b, D_2) = 1}} |S(a, bm, n, 1, D_1, D_2)|. 
\end{equation}
By the Weil bound, we have
\begin{multline}
  \mathcal{S}_{qh_1 > H_1} \leq  \sumprime_{h_2,  E_1, E_2} \sumprime_{q h_1 > H_1}  \sum_{d_1, d_2 | q} A(d_1, d_2, q)   \sum_{(n,q) = d_1} \sum_{(m,q) = d_2} |\alpha_m \beta_n| 
  \\
  \tau(E_1) \tau(E_2) (E_1 E_2)^{1/2}
 T(m,n,h_1, h_2).
\end{multline}
Trivially summing over $E_1$ and $E_2$ and using \eqref{eq:Abound}, we obtain
\begin{multline}
  \mathcal{S}_{qh_1 > H_1} \ll (X_1 X_2)^{3/2+\varepsilon} \sum_{q \leq \min(X_1, X_2)} \sum_{d_1, d_2 | q} \frac{ (d_1, d_2)}{q^2}   \sumprime_{h_1 > q^{-1} H_1} \sumprime_{h_2} \frac{1}{(h_1 h_2)^{3/2}}  
  \\
  \sum_{(n,q) = d_1} \sum_{(m,q) = d_2} |\alpha_m \beta_n| 
 T(m,n,h_1, h_2).
\end{multline}
Write the prime factorizations of $h_1$ and $h_2$ as follows:
\begin{gather}
\label{eq:h1h2primefactorizationdefinition}
\begin{split}
 h_1 = j_1 k_1 l_1, \quad j_1 = p_1 \dots p_r, \quad k_1= q_1^{b_1} \dots q_s^{b_s}, \quad l_1 =  \rho_1^{c_1} \dots \rho_t^{c_t}
 \\
 h_2 = j_2 k_2 l_2, \quad j_2= p_1^{a_1} \dots p_r^{a_r}, \quad k_2 = q_1 \dots q_s, \quad  l_2 =  \rho_1^{\gamma_1} \dots \rho_t^{\gamma_t},
 \end{split}
\end{gather}
where $a_i, b_i, c_i, \gamma_i \geq 2$, for all $i$, $(p_i, q_j \rho_k) = (q_j, \rho_k) = 1$ for all $i,j,k$, and all $p_i, q_j, \rho_k$ are prime.  

We first estimate $T(m,n,j_1, j_2)$.  Suppose $l \geq 2$.
By Lemmas \ref{lemma:Sevaluation} and \ref{lemma:KloostermanVanishes}, we have
\begin{equation}
 |S(a,bm, n, 1, p, p^l)| = |S(n,0;p) S(m,p;p^l)|. 
\end{equation}
If $p \nmid m$ then this vanishes, while if $p|m$ then we have
\begin{equation}
 |S(a,bm, n, 1, p, p^l)| = p |S(n,0;p) S(\tfrac{m}{p},1;p^{l-1})| \leq p (n,p) \delta(p|m) l p^{\frac{l-1}{2}}.
\end{equation}
Therefore,
\begin{equation}
 T(m,n,j_1,j_2) \ll (j_1 j_2)^{1/2+\varepsilon} (n,j_1) \delta(j_1 | m),
\end{equation}
and by symmetry,
\begin{equation}
 T(m,n,k_1, k_2) \ll (k_1 k_2)^{1/2+\varepsilon} (m,k_2) \delta(k_2 | n).
\end{equation}
Finally, by Lemma \ref{lemma:KloostermanUpperBound}, we have
\begin{equation}
T(m,n,l_1, l_2) \leq (l_1 l_2)^{1/2+\varepsilon} ((l_1, l_2)(mn, [l_1, l_2]))^{1/2}.
\end{equation}
Therefore, we have
\begin{multline}
\label{eq:SlargeH1Sum}
\mathcal{S}_{qh_1> H_1} \ll (X_1 X_2)^{3/2+\varepsilon}  \sum_{q \leq \min(X_1, X_2)} \sum_{d_1, d_2 | q} \frac{ (d_1, d_2)}{q^2}
\sumprime_{\substack{h_1 > q^{-1} H_1}} \frac{1}{h_1} \sumprime_{\substack{h_2 }} \frac{(l_1, l_2)^{1/2}  }{ h_2}
\\
\sum_{\substack{n \equiv 0 \shortmod{k_2} \\ (n,q) = d_1}} \sum_{\substack{m \equiv 0 \shortmod{j_1} \\ (m,q) = d_2}} |\alpha_m \beta_n| (mn, [l_1, l_2])^{1/2} (n, j_1) (m,k_2).
\end{multline}
We claim
\begin{equation}
\label{eq:BilinearSumBoundClaim}
 \sum_{\substack{n \equiv 0 \shortmod{k_2} \\ (n,q) = d_1}} \sum_{\substack{m \equiv 0 \shortmod{j_1} \\ (m,q) = d_2}} |\alpha_m \beta_n| (mn, [l_1, l_2])^{1/2} (n, j_1) (m,k_2) \ll \frac{N (X_1 X_2)^{\varepsilon}}{(d_1 d_2)^{1/2}} \| \alpha \| \| \beta \|.
\end{equation}
Toward this, we first observe the simple bound
\begin{equation}
\sum_{n \leq N} |\alpha_n| (n, q)^{1/2} 
\leq \tau(q) N^{1/2} \| \alpha \|.
\end{equation}
Using the trivial inequalities $(mn, [l_1, l_2]) \leq (m, l_1 l_2) (n, l_1 l_2)$, $(n,j_1) \leq (n, j_1)^{1/2} j_1^{1/2}$, $(m,k_2) \leq (m,k_2)^{1/2} k_2^{1/2}$, and the fact that $(j_1, k_2) = 1$, we derive the claim.

Inserting \eqref{eq:BilinearSumBoundClaim} into \eqref{eq:SlargeH1Sum}, we conclude
\begin{equation}
\mathcal{S}_{qh_1 > H_1} \ll (X_1 X_2)^{3/2+\varepsilon} N \| \alpha \| \| \beta \| 
\sum_{q \leq \min(X_1, X_2)} \sum_{d_1, d_2 | q} \frac{(d_1, d_2)}{(d_1 d_2)^{1/2} q^2} 
\sumprime_{h_1 > q^{-1} H_1} \frac{1}{h_1} \sumprime_{h_2} \frac{(l_1, l_2)^{1/2}}{h_2},
\end{equation}
and to complete the proof of Lemma \ref{lemma:LargeH1Bound} it now suffices to show
\begin{equation}
\label{eq:h1BigSum}
 \sumprime_{h_1 > H} \frac{1}{h_1} \sumprime_{h_2} \frac{(l_1, l_2)^{1/2}}{h_2} \ll H^{-1} (X_1 X_2)^{\varepsilon},
\end{equation}
where $H > 0$, using the easy estimate $(d_1, d_2) \leq (d_1 d_2)^{1/2}$, and trivially summing over $q$ (it is essentially a harmonic series).  

We now prove \eqref{eq:h1BigSum}. First we examine the inner sum over $h_2$.
Writing the expression in terms of the prime factorizations \eqref{eq:h1h2primefactorizationdefinition}, we have
\begin{equation}
\sumprime_{h_2} \frac{(l_1, l_2)^{1/2}}{h_2 } = \sum_{a_1, \dots, a_r, \gamma_1, \dots, \gamma_t \geq 2} \frac{ \rho_1^{\frac{\min(c_1, \gamma_1)}{2}} \dots \rho_t^{\frac{\min(c_t, \gamma_t)}{2}}}
{p_1^{a_1} \dots p_r^{a_r} q_1^{} \dots q_s^{} \rho_1^{\gamma_1} \dots \rho_t^{\gamma_t}}.
\end{equation}
The reader may recall that once $h_1$ is fixed, the prime divisors of $h_2$ are already determined, which explains why the sum is only over the exponents $a_i, \gamma_i$.
It is easily noted that
\begin{equation}
 \sum_{a_1, \dots, a_r \geq 2} \frac{1}
{p_1^{a_1} \dots p_r^{a_r} } \leq \frac{2^r}{(p_1 \dots p_r)^{2}}, \quad
 \sum_{\gamma_1, \dots, \gamma_t \geq 2} \frac{ \rho_1^{\frac{\min(c_1, \gamma_1)}{2}} \dots \rho_t^{\frac{\min(c_t, \gamma_t)}{2}}}
{\rho_1^{\gamma_1} \dots \rho_t^{\gamma_t}} \ll \frac{2^t}{\rho_1 \dots \rho_t},
\end{equation}
with an absolute implied constant,
and so,
\begin{equation}
\frac{1}{h_1} \sumprime_{h_2} \frac{(l_1, l_2)^{1/2}}{h_2 }
\ll \frac{2^{r+t}}{(p_1 \dots p_r)^{3} q_1^{ b_1+1} \dots q_s^{b_s+1} \rho_1^{c_1 + 1} \dots \rho_t^{c_t + 1}}.
\end{equation}
Inserting this into the left hand side of \eqref{eq:h1BigSum}, and recalling the implicit condition $h_1 \leq X_1$, we have
\begin{multline}
\sumprime_{h_1 > H} \frac{1}{h_1} \sumprime_{h_2} \frac{(l_1, l_2)^{1/2}}{h_2} \ll  \sum_{p_1 \dots p_r q_1^{b_1} \dots q_s^{b_s} \rho_1^{c_1} \dots \rho_t^{c_t} > H} \frac{2^{r+t}}{(p_1 \dots p_r)^{3} q_1^{ b_1+1} \dots q_s^{b_s+1} \rho_1^{c_1 + 1} \dots \rho_t^{c_t + 1}} 
 \\
 \leq 
\frac{1}{H}  \sum_{p_1 \dots p_r q_1^{b_1} \dots q_s^{b_s} \rho_1^{c_1} \dots \rho_t^{c_t} > H}  \frac{2^{r+t}}{(p_1 \dots p_r)^{2} q_1^{} \dots q_s^{} \rho_1^{} \dots \rho_t^{}} \ll H^{-1} X_1^{\varepsilon}.
\end{multline}
This shows \eqref{eq:h1BigSum}, and concludes the proof of \eqref{eq:LargeH1Bound1}.  The other estimate \eqref{eq:LargeH1Bound2} follows from \eqref{eq:LargeH1Bound1} by symmetry.
\end{proof}

\subsection{Small $h_i$}
The main goal of this subsection is
\begin{mylemma}
\label{lemma:SmallH1Bound}
We have
\begin{equation}
 \mathcal{S}_{qh_i \leq H_i} \ll (X_1 H_2 + X_2 H_1) (X_1 X_2 N)^{\varepsilon} M^*(\alpha)^{1/2} M^*(\beta)^{1/2}.
\end{equation}
\end{mylemma}
Choosing $H_1 = X_1$, $H_2 = X_2$, we have $M^*(\beta) = M(\beta)$ and $M^*(\alpha) = M(\alpha)$, and $\mathcal{S}_{qh_1 > X_1} = \mathcal{S}_{qh_2 > X_2} = 0$, and we obtain Theorem \ref{thm:BilinearKloosterman}.  More generally, combining Lemmas \ref{lemma:LargeH1Bound} and \ref{lemma:SmallH1Bound} proves Theorem \ref{thm:BilinearKloostermanHversion}.

\begin{proof}
We begin by inserting 
the formula
\begin{multline}
 S(\overline{q} \overline{E_1}^2 E_2, \overline{q} \overline{E_2}^2 E_1 m, n, 1,  h_1, h_2) 
\\ 
 = 
 \sum_{t \shortmod{h_1}} \sum_{u \shortmod{h_2}} \widehat{S}(\overline{q} \overline{E_1}^2 E_2, q E_2^2 \overline{E_1} u, t, 1, h_1, h_2) e\Big(\frac{tn}{h_1}\Big) e\Big(\frac{um}{h_2} \Big),
\end{multline}
into \eqref{eq:SboundSeparatedSomewhat}, getting
\begin{multline}
  \mathcal{S}_{qh_i \leq H_i} \leq \sumprime_{h_1, h_2, q, E_1, E_2} \sum_{d_1, d_2 | q} A(d_1, d_2, q) \sum_{t \shortmod{h_1}} \sum_{u \shortmod{h_2}} |\widehat{S}(\overline{q} \overline{E_1}^2 E_2, q E_2^2 \overline{E_1} u, t, 1, h_1, h_2) |  
  \\
  \Big|  
  \sum_{(n,q) = d_1} \beta_n S(E_2 \overline{q} \overline{h_1}^2 h_2, n, E_1) e\Big(\frac{tn}{h_1}\Big)
  \sum_{(m,q) = d_2} \alpha_m 
  S(E_1 \overline{q} \overline{h_2}^2 h_1, m, E_2)  e\Big(\frac{um}{h_2} \Big)
 \Big|.
\end{multline}
We shall occasionally leave the conditions $qh_1 \leq H_1$, $qh_2 \leq H_2$ implicit in the notation.
By Cauchy's inequality, we write 
 $\mathcal{S}_{qh_i \leq H_i} \leq S_1^{1/2} S_2^{1/2}$ where
\begin{multline}
  S_1 = \sumprime_{h_1, h_2, q, E_1, E_2} \sum_{d_1, d_2 | q} A(d_1, d_2, q) \sum_{t \shortmod{h_1}} \sum_{u \shortmod{h_2}} |\widehat{S}(\overline{q} \overline{E_1}^2 E_2, q E_2^2 \overline{E_1} u, t, 1, h_1, h_2) |  
  \\
  \Big|  
  \sum_{(n,q) = d_1} \beta_n S(E_2 \overline{q} \overline{h_1}^2 h_2, n, E_1) e\Big(\frac{tn}{h_1}\Big)
 \Big|^2,
\end{multline}
and $S_2$ is given by a similar formula.  Write this as
\begin{equation}
 S_1 = \sumprime_{q, E_1}  \sum_{d_1, d_2 |q} A(d_1, d_2, q) S_1',
\end{equation}
where
\begin{multline}
  S_1' = \sumprime_{h_1, h_2,  E_2}  \sum_{t \shortmod{h_1}} \sum_{u \shortmod{h_2}} |\widehat{S}(\overline{q} \overline{E_1}^2 E_2, q E_2^2 \overline{E_1} u, t, 1, h_1, h_2) |  
  \\
  \Big|  
  \sum_{(n,q) = d_1} \beta_n S(E_2 \overline{q} \overline{h_1}^2 h_2, n, E_1) e\Big(\frac{tn}{h_1}\Big)
 \Big|^2.
\end{multline}
Recalling the definition of $\mathcal{R}$ from \eqref{eq:Rdefinition}, we have
\begin{equation}
 S_1' \leq \sumprime_{h_1, h_2,  E_2}  \sum_{t \shortmod{h_1}} \mathcal{R}(t,h_1, h_2) 
  \Big|  
  \sum_{(n,q) = d_1} \beta_n S(E_2 \overline{q} \overline{h_1}^2 h_2, n, E_1) e\Big(\frac{tn}{h_1}\Big)
 \Big|^2.
\end{equation}

Using the trick described surrounding \eqref{eq:KloostermanProductCompleteSum}, we extend the sum over $E_2$ 
to a complete sum modulo $E_1 \leq \frac{X_1}{q h_1}$ (forgetting the various coprimality conditions on $E_2$ by positivity), giving
\begin{equation}
 S_1' \leq \sumstar_{x \shortmod{E_1}}  \sumprime_{h_1, h_2}   \sum_{t \shortmod{h_1}} \mathcal{R}(t,h_1, h_2)  \Big(\frac{X_1}{q h_1} + \frac{X_2}{q h_2}\Big) 
  \Big|  
  \sum_{(n,q) = d_1} \beta_n e\Big(\frac{xn}{E_1}\Big) e\Big(\frac{tn}{h_1}\Big)
 \Big|^2.
\end{equation}
Next write $g_1 = (t, h_1)$, and change variables $h_1 = g_1 h_1'$, $t = g_1 t'$, so that $(t', h_1') = 1$.  This gives
\begin{equation}
 S_1' \leq \sumstar_{x \shortmod{E_1}}  \sumprime_{g_1, h_1', h_2}   
 \thinspace
 \sumstar_{t' \shortmod{h_1'}} \mathcal{R}(t' g_1 ,g_1 h_1', h_2)  \Big(\frac{X_1}{q g_1 h_1'} + \frac{X_2}{q h_2}\Big) 
  \Big|  
  \sum_{(n,q) = d_1} \beta_n e\Big(\frac{xn}{E_1}\Big) e\Big(\frac{t'n}{h_1'}\Big)
 \Big|^2.
\end{equation}
Here the primes on the sums refer to the conditions \eqref{eq:primesumconditions2} with $h_1$ replaced by $g_1 h_1'$.
Observe $\mathcal{R}(t' g_1, h_1' g_1, h_2) = \mathcal{R}(g_1, h_1' g_1, h_2)$, by definition, and
re-arrange this in the form
\begin{multline}
\label{eq:S1'bound}
 S_1' \leq \sumstar_{x \shortmod{E_1}}
 \sumprime_{h_1' \leq \frac{H_1}{q}} \thinspace \sumstar_{t' \shortmod{h_1'}} \Big|  
  \sum_{(n,q) = d_1} \beta_n e\Big(\frac{xn}{E_1}\Big) e\Big(\frac{t'n}{h_1'}\Big)
 \Big|^2
 \\
 \sumprime_{g_1 \leq \frac{H_1}{ qh_1' } } \sumprime_{\substack{h_2 \leq \frac{H_2}{q} } }    \mathcal{R}(g_1, g_1 h_1', h_2)  \Big(\frac{X_1}{q g_1 h_1'} + \frac{X_2}{q h_2}\Big).
\end{multline}

We claim
\begin{equation}
\label{eq:RSumClaimedBound}
 \sumprime_{g_1 \leq \frac{H_1}{ qh_1' } } \sumprime_{\substack{h_2 \leq \frac{H_2}{q} } }     \mathcal{R}(g_1, g_1 h_1', h_2)  \Big(\frac{X_1}{q g_1 h_1'} + \frac{X_2}{q h_2}\Big) \ll \frac{(X_1 X_2)^{\varepsilon}}{q^2 h_1'} (X_1 H_2 + X_2 H_1).
\end{equation}

\begin{proof}[Proof of claim]
We first bound the sum with the factor $\frac{X_1}{q g_1 h_1'}$. 
By Corollary \ref{coro:Rbound}, we have
\begin{equation}
 \sumprime_{g_1 \leq \frac{H_1}{ q h_1' } } \sumprime_{\substack{h_2 \leq \frac{H_2}{q} } }    \mathcal{R}(g_1, g_1 h_1', h_2)  \frac{X_1}{q g_1 h_1'} 
 \ll (X_1 X_2)^{\varepsilon}
 \sumprime_{g_1 \leq \frac{H_1}{q h_1' } } \sumprime_{\substack{h_2 \leq \frac{H_2}{q} } }    h_2  \frac{X_1}{q g_1 h_1'} \sum_{\substack{d | g_1 \\ d^3 | (g_1 h_1', h_2)^2}} d.
\end{equation}
Reversing the order of summation, and estimating the sum over $h_2$ by Lemma \ref{lemma:divisortypebound} (one may safely drop the condition $d^3 | h_2^2$ when summing over $h_2$),  this is
\begin{equation}
\label{eq:claimboundwithX1}
 \ll \frac{(X_1  X_2)^{\varepsilon} X_1 H_2}{q^2 h_1'} 
 \sum_{d \leq X_1} d
 \sum_{\substack{g_1 \leq \frac{H_1}{ h_1'} \\ d | g_1, \thinspace  d^3 | (g_1 h_1')^2}}  \frac{1}{g_1}.  
\end{equation}
Next we write $g_1 = dr$, where now $d | r^2 h_1'^2$, so we may execute the sum over $d$ first as a divisor sum, and finally the sum over $r$ satisfies $\sum_{r \leq X_1} r^{-1+\varepsilon} \ll X_1^{\varepsilon}$.  This immediately gives a bound consistent with \eqref{eq:RSumClaimedBound}.

For the second sum with $\frac{X_2}{q h_2}$, we have by Corollary \ref{coro:Rbound} that 
\begin{equation}
\label{eq:SomeEquationInProofOfClaim}
 \sumprime_{g_1 \leq \frac{H_1}{q h_1' } } \sumprime_{\substack{h_2 \leq \frac{H_2}{q} } }    \mathcal{R}(g_1, g_1 h_1', h_2) \frac{X_2}{q h_2} 
 \ll 
(X_1 X_2)^{\varepsilon} \sumprime_{g_1 \leq \frac{H_1}{q h_1' } } \sumprime_{\substack{h_2 \leq \frac{H_2}{q} } } \frac{X_2}{q} \sum_{\substack{d | g_1 \\ d^3 | (g_1 h_1', h_2)^2}} d.
\end{equation}
We shall reverse the order of summation and execute the sum over $h_2$ first.  
The sum over $h_2$ is bounded by $O(X_2^{\varepsilon})$ using Lemma \ref{lemma:divisortypebound}, because 
one of the summation conditions is that $h_2$ and $g_1 h_1'$ share the same prime factors.  
Then the right hand side of \eqref{eq:SomeEquationInProofOfClaim} is
\begin{equation}
 \ll (X_1 X_2)^{\varepsilon} \frac{X_2}{q} \sum_{g_1 \leq \frac{H_1}{q h_1' } } \sum_{\substack{d|g_1\\ d^3 | (g_1 h_1')^2}} d.
\end{equation}
Write $g_1 = dr$, whence this is
\begin{equation}
 \ll (X_1 X_2)^{\varepsilon} \frac{X_2}{q} \sum_{d \leq X_1} d \sum_{\substack{r \leq \frac{H_1}{d qh_1'} \\ r^2 h_1'^2 \equiv 0 \shortmod{d}}} 1.
\end{equation}
Suppose that $d = d' f$ where $d'$ consists of the prime powers corresponding to the primes that divide $h_1'$, so that $f := d/d'$ is then coprime to $h_1'$.  Then we have $r^2 \equiv 0 \pmod{f}$.  Let $f^*$ be the integer such that the congruence $r^2 \equiv 0 \pmod{f}$ is equivalent to $r \equiv 0 \pmod{f^*}$.  Then the above expression is bounded by
\begin{equation}
 \ll  (X_1 X_2)^{\varepsilon} \frac{X_2 H_1}{q^2 h_1'} \sum_{d'} \sum_{f} \frac{1}{f^*} \ll (X_1 X_2)^{\varepsilon} \frac{X_2 H_1}{q^2 h_1'},
\end{equation}
since Lemma \ref{lemma:divisortypebound} shows the sum over $d'$ is $\ll (X_1 X_2)^{\varepsilon}$, and the sum over $f$ is $\ll (X_1 X_2)^{\varepsilon}$ by elementary reasoning (e.g. Rankin's trick).
Thus we arrive at a bound consistent with \eqref{eq:RSumClaimedBound}.
\end{proof}

The claim \eqref{eq:RSumClaimedBound} applied to \eqref{eq:S1'bound} implies
\begin{equation}
 S_1' \ll \frac{(X_1 X_2)^{\varepsilon}}{q^2} (X_1 H_2 + X_2 H_1)  \sumstar_{x \shortmod{E_1}}  \sumprime_{h_1' \leq H_1}  \frac{1}{h_1'} \thinspace \sumstar_{t' \shortmod{h_1'}} 
  \Big|  
  \sum_{(n,q) = d_1} \beta_n e\Big(\frac{xn}{E_1}\Big) e\Big(\frac{t'n}{h_1'}\Big)
 \Big|^2.
\end{equation}
Inserting this into $S_1$, and using the Chinese remainder theorem to combine 
the sums modulo $E_1$ and $h_1'$ to the single modulus $E_1 h_1'$, we derive
\begin{multline}
 S_1 \ll (X_1 H_2 + X_2 H_1) (X_1 X_2)^{\varepsilon} \sum_{q \leq \min(H_1, H_2)} \sum_{d_1, d_2 | q} \frac{A(d_1, d_2, q)}{q^2} 
 \\
 \sum_{\substack{E_1 h_1' \leq \frac{X_1}{q} \\ h_1' \leq H_1}} \frac{1}{h_1'} 
 \thinspace
 \sumstar_{x \shortmod{E_1 h_1'}} \Big|  
  \sum_{(n,q) = d_1} \beta_n e\Big(\frac{xn}{E_1 h_1'}\Big)
 \Big|^2.
\end{multline}
We group together $E_1 h_1'$ into a single variable $c$, use $h_1'^{-1} \leq 1$ and \eqref{eq:Abound}, obtaining
\begin{equation}
 S_1 \ll (X_1 H_2 + X_2 H_1) (X_1 X_2)^{\varepsilon} \sum_{q \leq \min(H_1, H_2)} \sum_{d_1, d_2 | q} \frac{(d_1, d_2)^3}{q d_1 d_2} \sum_{\substack{c \leq \frac{X_1}{q} \\ (c,q) = 1}} \thinspace \sumstar_{x \shortmod{c}} 
 \Big|  
  \sum_{(n,q) = d_1} \beta_n e\Big(\frac{xn}{c}\Big)
 \Big|^2.
\end{equation}
Using the crude bound $(d_1 d_2)^{-1}  (d_1, d_2)^3 \leq d_1$, and trivially summing over $d_2$, we obtain
\begin{equation}
\label{eq:S1bound}
 S_1 \ll (X_1 H_2 + X_2 H_1) (X_1 X_2)^{\varepsilon} \sum_{q \leq \min(H_1, H_2)} \sum_{d_1| q} \frac{d_1}{q} \sum_{\substack{c \leq \frac{X_1}{q} \\ (c,q) = 1}} \thinspace  \sumstar_{x \shortmod{c}} 
 \Big|  
  \sum_{(n,q) = d_1} \beta_n e\Big(\frac{xn}{c}\Big)
 \Big|^2,
\end{equation}
which is $(X_1 H_2 + X_2 H_1) (X_1 X_2)^{\varepsilon} M^*(\beta)$
where recall $M(\beta)$ was defined by \eqref{eq:MbetaDef} and $M^*(\beta)$ has the same definition but with $q \leq \min(H_1, H_2)$.

We also need to estimate $S_2$.  It is given by a similar formula to $S_1$, except with $h_1$ and $h_2$ switched, $E_1$ and $E_2$ switched, $\beta_n$ replaced by $\alpha_m$, $d_1$ and $d_2$ switched, and we need to work with $\mathcal{R}'(u, h_1, h_2)$ instead of $\mathcal{R}$ (for which see Lemma \ref{lemma:R'property}).  Therefore, by a symmetry argument, we have $S_2 \ll (X_1 H_2 + X_2 H_1) (X_1 X_2)^{\varepsilon} M^*(\alpha)$.
\end{proof}

\begin{myremark}
 \label{remark:bilinearwithKloostermanTheoremVariant}
 The proof given above works equally well if we replace $S(1,m,n,1,D_1, D_2)$ by $S(1, \epsilon_1 m, \epsilon_2 n, 1, D_1, D_2)$ for $\epsilon_1, \epsilon_2 \in \{ -1, 1 \}$.
\end{myremark}

%

\section{Bounds on $\mathcal{R}$}
\label{section:Rbound}
This section is devoted to the long proof of Lemma \ref{lemma:Rbound}.
The overarching idea of the proof is to evaluate $\widehat{S}(a,u,t,b, p^k, p^l)$ in explicit terms (as much as possible), and to trivially sum over $u$.  Lemma \ref{lemma:ShatReverseModuli} will allow us to focus almost entirely on the case $k < l$.  Except for the cases $k = l \geq 2$, we have evaluated $\widehat{S}$ exactly.  It is a pleasant fact that this is much easier than evaluating the Kloosterman sum itself (compare to Theorem 0.3 of \cite{DabrowskiFisher}).

In the proof of Theorem \ref{thm:BilinearKloosterman} we only needed estimates on $\mathcal{R}(t, p^k, p^l)$ when $k, l \geq 1$ and $\max(k,l) \geq 2$, but since the small values of $k$ and $l$ are easily treated, we shall cover all the cases as stated in Lemma \ref{lemma:Rbound}.

\subsection{The case $k=0$, or $l =0$}
By a direct calculation, and using \eqref{eq:KloostermanFactorizationCoprimeModuli}, we have
\begin{equation}
 \widehat{S}(a,u,t,b, 1, p^l) = e\Big(\frac{\overline{u} b}{p^l} \Big) \delta(p \nmid u),
\end{equation}
and by symmetry (that is, Lemma \ref{lemma:ShatReverseModuli}),
\begin{equation}
 \widehat{S}(a,u,t,b, p^k, 1) = e\Big(\frac{\overline{t} a}{p^k} \Big) \delta(p \nmid t).
\end{equation}
Trivially summing over $u$, we easily derive $\mathcal{R}(t,p^k,p^l) \leq p^l$ in case $k=0$ or $l=0$.

\subsection{The case $k=l=1$}
By Lemma \ref{lemma:Sevaluation}, 
\begin{equation}
\label{eq:KloostermanEvaluationPrimePrimePower}
 S(a,y,x,b,p,p^l) = S(x,0;p) S(y,bp;p^l) + S(a,0;p) S(b,yp;p^l) + (p-1)\delta(l=1).
\end{equation}
Therefore, recalling $(ab,p) = 1$ and $S(a,0;p) = -1 = S(b,0;p)$, we have
\begin{multline}
 \widehat{S}(a,u,t,b,p,p) = \frac{1}{p^{2}} \sum_{x \shortmod{p}} \sum_{y \shortmod{p}} e\Big(\frac{-x t}{p}\Big) e\Big(\frac{-y u}{p}\Big) [
 S(x,0;p) S(y,0;p) + p]
 \\
 = \delta(p\nmid t) \delta(p \nmid u) +  p\delta(p|t) \delta(p | u).
\end{multline}
We immediately deduce $\mathcal{R}(t,p,p) \leq p$.

\subsection{The case $k=1, l \geq 2$}
Using \eqref{eq:KloostermanEvaluationPrimePrimePower} and the fact that 
$S(b,yp;p^l) = 0$ following from Lemma \ref{lemma:KloostermanVanishes}, we derive
\begin{multline}
\label{eq:ShatFormulaPrimePrimePower}
 \widehat{S}(a, u, t, b, p, p^l) = \frac{1}{p^{l+1}} \sum_{x \shortmod{p}} \sum_{y \shortmod{p^l}} S(x,0;p) S(y,bp;p^l) e\Big(\frac{-x t}{p}\Big) e\Big(\frac{-y u}{p^l}\Big)
 \\
 = \delta(p\nmid t) \delta(p \nmid u) e\Big(\frac{\overline{u} bp}{p^l} \Big).
\end{multline}
We conclude
\begin{equation}
 \mathcal{R}(t,p,p^l) \leq \delta(p \nmid t) p^l.
\end{equation}

\subsection{The case $l=1, k \geq 2$}
By \eqref{eq:ShatFormulaPrimePrimePower} and Lemma \ref{lemma:ShatReverseModuli}, we have
\begin{equation}
 \widehat{S}(a,u,t,b,p^k,p) = \delta(p\nmid t) \delta(p \nmid u) e\Big(\frac{\overline{t} ap}{p^k} \Big),
\end{equation}
and so $\mathcal{R}(t,p^k,p) \leq p$.

Remark. For the remaining cases we do not use direct evaluations of $S$ and instead calculate $\widehat{S}$ from the definition.  We have
\begin{multline}
 \widehat{S}(a,u,t,b,p^k,p^l) = 
 \frac{1}{p^{k+l}} \sum_{x \shortmod{p^k}}
 \sum_{y \shortmod{p^l}}
S(a, y, x, b, p^k, p^l) e\Big(\frac{-x t}{p^k}\Big) e\Big(\frac{-y u}{p^l}\Big)
\\
= 
\Big[\frac{1}{p^{k+l}} \sum_{x \shortmod{p^k}}
 \sum_{y \shortmod{p^l}} 
 \mathop{\sum_{\substack{B_1, C_1 \shortmod{p^k} \\ (B_1, C_1, p^k) = 1}} 
 \sum_{\substack{B_2, C_2 \shortmod{p^l} \\ (B_2, C_2, p^l) = 1}}
 }_{\substack{p^k C_2 + B_1 B_2 + p^l C_1 \equiv 0 \shortmod{p^{k+l}}
 \\ Y_1 B_1 + Z_1 C_1 \equiv 1 \shortmod{p^k}
 \\ Y_2 B_2 + Z_2 C_2 \equiv 1 \shortmod{p^l}
 }} e\Big(\frac{a B_1 + x(Y_1 p^l - Z_1 B_2)}{p^k} \Big)
 \\
 e\Big(\frac{y B_2 + b(Y_2 p^k - Z_2 B_1)}{p^l} \Big) e\Big(\frac{-x t}{p^k}\Big) e\Big(\frac{-y u}{p^l}\Big) \Big].
\end{multline}
This simplifies as
\begin{equation}
\label{eq:ShatFormulaGeneral}
\widehat{S}(a,u,t,b,p^k,p^l) =  \mathop{\sum_{\substack{B_1, C_1 \shortmod{p^k} \\ (B_1, C_1, p^k) = 1}} 
 \sum_{\substack{ C_2 \shortmod{p^l} \\ (u, C_2, p^l) = 1}}
 }_{\substack{p^k C_2 + B_1 u + p^l C_1 \equiv 0 \shortmod{p^{k+l}}
 \\ Y_1 p^l - Z_1 u \equiv t \shortmod{p^k}
 \\ Y_1 B_1 + Z_1 C_1 \equiv 1 \shortmod{p^k}
 \\ Y_2 u + Z_2 C_2 \equiv 1 \shortmod{p^l}
 }} e\Big(\frac{a B_1}{p^k} \Big)
 e\Big(\frac{ b(Y_2 p^k - Z_2 B_1)}{p^l} \Big).
\end{equation}
Although a large expression, we found it helpful to have all the conditions written in the summation sign.

\subsection{The case $l > k \geq 2$}
Suppose that $p^{\nu} || t$, and write $t = p^{\nu} t'$.  
Here we will show
\begin{equation}
\label{eq:ShatPrimePowerPrimePowerEvaluation}
 \widehat{S}(a,u,t,b,p^k,p^l) = p^{\nu} e\Big(\frac{b\overline{u'}}{p^{l-k+\nu}} \Big) e\Big(\frac{a\overline{t'} p^{l-k}}{p^{\nu}} \Big) S(a, b\overline{t'u'}, p^{\nu}) \delta(\nu \leq k/2),
\end{equation}
where the sum vanishes unless $p^{\nu} || u$, in which case we write  $u = p^{\nu} u'$.

Using only the trivial bound (not even the Weil bound) for the Kloosterman sum, we conclude
\begin{equation}
 \mathcal{R}(t,p^k,p^l) \leq p^{\nu} \sum_{u \shortmod{p^l}, p^{\nu} |u} p^{\nu} = p^{\nu + l},
\end{equation}
and in addition we have $\nu \leq k/2$.  This estimate is consistent with Lemma \ref{lemma:Rbound}.

The congruence $p^k C_2 + B_1 u + p^l C_1 \equiv 0 \pmod{p^{k+l}}$ implies $p^k | B_1 u$, and is equivalent to 
\begin{equation}
\label{eq:C2congruence}
 C_2 \equiv -\frac{B_1 u}{p^k} - p^{l-k} C_1 \pmod{p^l}.
\end{equation}
Suppose $p^{k_1} || B_1$, and write $\widehat{S} = \sum_{k_1=0}^{k} V_{k_1}$ correspondingly.  We first evaluate the terms with $1 \leq k_1 \leq k-1$.  We write $B_1 = p^{k_1} R_1$ where $R_1$ runs mod $p^{k_2}$, with $k_1 + k_2 = k$.  We also have $p^{k_2} | u$.  Since $p | B_1$ and $p | u$, the coprimality conditions now require $p \nmid C_1$ and $p \nmid C_2$.  If $p^{k_2+1} | u$ then \eqref{eq:C2congruence} would imply $p | C_2$, a contradiction.  So $p^{k_2} || u$, and we write $u = p^{k_2} u'$.  We set $Y_1 = Y_2 = 0$, and $Z_i = \overline{C_i}$.  Then we have
\begin{equation}
 V_{k_1} = \mathop{ \sumstar_{R_1 \shortmod{p^{k_2}}} \thinspace \sumstar_{C_1 \shortmod{p^k}} \thinspace
 \sumstar_{ C_2 \shortmod{p^l}}
 }_{\substack{C_2 \equiv -R_1 u' - p^{l-k} C_1 \pmod{p^l}
 \\ - \overline{C_1} u' p^{k_2} \equiv t \shortmod{p^k}
 }} e\Big(\frac{a R_1}{p^{k_2}} \Big)
 e\Big(\frac{ -b \overline{C_2} R_1}{p^{l-k_1}} \Big).
\end{equation}
Next we observe that $p^{k_2} || t$, so we write $t = p^{k_2} t'$, and then we have $C_1 \equiv - u' \overline{t'} \pmod{p^{k_1}}$.  
With these evaluations, we have
\begin{equation}
 V_{k_1} = \sumstar_{R_1 \shortmod{p^{k_2}}} \sumstar_{\substack{C_1 \shortmod{p^k} \\ C_1 \equiv - u' \overline{t'} \shortmod{p^{k_1}}}}
 e\Big(\frac{a R_1}{p^{k_2}} \Big)
 e\Big(\frac{ b R_1\overline{(R_1 u' + p^{l-k} C_1)} }{p^{l-k_1}} \Big).
\end{equation}
To help simplify this expression, we expand as follows:
\begin{equation}
 e\Big(\frac{ b R_1\overline{(R_1 u' + p^{l-k} C_1)} }{p^{l-k_1}} \Big) = 
 e\Big(\frac{ b \overline{u'} (R_1 u' + p^{l-k} C_1 - p^{l-k} C_1) \overline{(R_1 u' + p^{l-k} C_1)} }{p^{l-k_1}} \Big).
\end{equation}
After simplification, this gives
\begin{equation}
V_{k_1} = e\Big(\frac{b\overline{u'}}{p^{l-k_1}} \Big) \sumstar_{R_1 \shortmod{p^{k_2}}} \sumstar_{\substack{C_1 \shortmod{p^k} \\ C_1 \equiv - u' \overline{t'} \shortmod{p^{k_1}}}}
 e\Big(\frac{a R_1}{p^{k_2}} \Big)  e\Big(\frac{- b \overline{u'}  C_1 \overline{(R_1 u' + p^{l-k} C_1)} }{p^{k_2}}\Big).
\end{equation}

We claim that if $1\leq k_1 < k_2 \leq k-1$, then $V_{k_1}=0$.
For this, write $C_1 = f_1 + p^{k_2-1} f_2$ with $f_1$ running mod $p^{k_2 - 1}$ and $f_2$ mod $p^{k_1  + 1}$.  Since $k_1 < k_2$, the congruence $C_1 \equiv - u' \overline{t'} \pmod{p^{k_1}}$ gives no condition on $f_2$.
Then note that since $l-k + k_2 - 1 \geq k_2$, we have
\begin{equation}
 e\Big(\frac{- b \overline{u'}  (f_1 + p^{k_2-1} f_2) \overline{(R_1 u' + p^{l-k} (f_1 + p^{k_2-1} f_2))} }{p^{k_2}}\Big) = e\Big(\frac{- b \overline{u'}  (f_1 + p^{k_2-1} f_2) \overline{(R_1 u' + p^{l-k} f_1)} }{p^{k_2}}\Big),
\end{equation}
and so the sum over $f_2$ will cause $V_{k_2}$ to vanish.  

Now suppose $k_1 \geq k_2$.  Then the congruence on $C_1$ mod $p^{k_1}$ determines $C_1$ mod $p^{k_2}$, and hence
\begin{equation}
 V_{k_1} = e\Big(\frac{b\overline{u'}}{p^{l-k_1}} \Big) \sumstar_{R_1 \shortmod{p^{k_2}}} \sumstar_{\substack{C_1 \shortmod{p^k} \\ C_1 \equiv - u' \overline{t'} \shortmod{p^{k_1}}}}
 e\Big(\frac{a R_1}{p^{k_2}} \Big)
  e\Big(\frac{ b  \overline{t'} \overline{(R_1 u' + p^{l-k} (-u' \overline{t'}))} }{p^{k_2}}\Big),
\end{equation}
which simplifies as
\begin{equation}
 V_{k_1} = p^{k_2} e\Big(\frac{b\overline{u'}}{p^{l-k_1}} \Big) \sumstar_{R_1 \shortmod{p^{k_2}}}
 e\Big(\frac{a R_1}{p^{k_2}} \Big)
  e\Big(\frac{ b  \overline{t' u'} \overline{(R_1 - p^{l-k}  \overline{t'})} }{p^{k_2}}\Big).
\end{equation}
Changing variables $R_1 \rightarrow R_1 + p^{l-k} \overline{t'}$, we have
\begin{equation}
\label{eq:Vk1evaluation}
 V_{k_1} = p^{k_2} e\Big(\frac{b\overline{u'}}{p^{l-k_1}} \Big) e\Big(\frac{a\overline{t'} p^{l-k}}{p^{k_2}} \Big) S(a, b\overline{t'u'}, p^{k_2}),
\end{equation}
and we recollect that $p^{k_2} || u$, $p^{k_2} || t$, and $1 \leq k_2 \leq k_1 \leq k-1$.  If we define $\nu$ by $p^{\nu} || t$, then we have
\begin{equation}
\label{eq:Vk1SumEvaluation}
 \sum_{k_1=1}^{k-1} V_{k_1} = p^{\nu} e\Big(\frac{b\overline{u'}}{p^{l-k+\nu}} \Big) e\Big(\frac{a\overline{t'} p^{l-k}}{p^{\nu}} \Big) S(a, b\overline{t'u'}, p^{\nu}) \delta(1 \leq \nu \leq k/2).
\end{equation}

So far we have left the cases with $k_1 =0 $ and $k_1 = k$ unevaluated, so we next turn to this.  

We claim $V_0=0$, which takes some calculation.
We have $p \nmid B_1$ so we can set $Y_1 = \overline{B_1}$ and $Z_1 = 0$.  If $p^{k+1} | u$ then \eqref{eq:C2congruence} would mean $p | C_2$, a contradiction.  So $p^{k} || u$ and we write $u = p^{k} u'$.  We must have $p \nmid C_2$ so we set $Y_2 = 0$, $Z_2 = \overline{C_2}$.  With these evaluations, we have $p^k | t$ and 
\begin{equation}
V_0 =
\sumstar_{B_1 \shortmod{p^k}} \sum_{C_1 \shortmod{p^k}} \sum_{\substack{ C_2 \shortmod{p^l} \\ C_2 \equiv - B_1 u' - p^{l-k} C_1 \pmod{p^l}}}
e\Big(\frac{a B_1}{p^k} \Big)
 e\Big(\frac{ - b\overline{C_2} B_1}{p^l} \Big).
\end{equation}
Now we can write $C_1 = f_1 + p^{k-1} f_2$, and 
$C_2 = - B_1 u' - p^{l-k} f_1 - p^{l-1} f_2$, and so
\begin{equation}
 V_0 = \sumstar_{B_1 \shortmod{p^k}} e\Big(\frac{a B_1}{p^k} \Big) \sum_{f_1 \shortmod{p^{k-1}}} \sum_{f_2 \shortmod{p}} 
 e\Big(\frac{ - b\overline{(- B_1 u' - p^{l-k} f_1 - p^{l-1} f_2)} B_1}{p^l} \Big).
\end{equation}
Note $\overline{1+ p^{l-1} f_2} \equiv 1 - p^{l-1} f_2 \pmod{p^l}$ since $l \geq 2$.  This shows that the sum over $f_2$ vanishes, as desired.
%

Finally we evaluate $V_k$.  
Then we have $p^{k} | B_1$ so may set $B_1 = 0$ (that is, we choose the integer $0$ for the coset representative of $0 \pmod{p^k}$), $p \nmid C_1$, $p | C_2$ so $p \nmid u$, and we may set $Y_1 = 0$, $Z_1 = \overline{C_1}$, $Y_2 = \overline{u}$,  $Z_2 = 0$.  Then 
\begin{equation}
 V_k = e\Big(\frac{b p^k \overline{u}}{p^l} \Big) \sumstar_{\substack{C_1 \shortmod{p^k} \\ - \overline{C_1} u \equiv t \shortmod{p^k} }} \sum_{\substack{C_2 \shortmod{p^l} \\ C_2 \equiv - p^{l-k} C_1 \shortmod{p^l}}} 1.
\end{equation}
This means $p \nmid t$ and both $C_1$ and $C_2$ are uniquely determined.  Therefore,
\begin{equation}
\label{eq:VkSumEvaluation}
 V_k = \delta(p \nmid t) \delta(p \nmid u) e\Big(\frac{b p^k \overline{u}}{p^l} \Big),
\end{equation}
which coincidentally agrees with the right hand side of \eqref{eq:Vk1evaluation}, except with $k_2 = \nu = 0$.  Therefore, by adding \eqref{eq:Vk1SumEvaluation} and \eqref{eq:VkSumEvaluation} we obtain \eqref{eq:ShatPrimePowerPrimePowerEvaluation}, as desired.

%

\subsection{The case $k > l \geq 2$}
By \eqref{eq:ShatPrimePowerPrimePowerEvaluation} and Lemma \ref{lemma:ShatReverseModuli}, we have
\begin{equation}
\label{eq:ShatPrimePowerPrimePowerEvaluation2}
 \widehat{S}(a,u,t,b,p^k,p^l) = p^{\nu} e\Big(\frac{a\overline{t'}}{p^{k-l+\nu}} \Big) e\Big(\frac{b\overline{u'} p^{k-l}}{p^{\nu}} \Big) S(b, a\overline{t'u'}, p^{\nu}) \delta(\nu \leq l/2),
\end{equation}
where again $\nu$ is defined by $p^{\nu} || u$ and $p^{\nu} || t$.
Using only the trivial bound for the Kloosterman sum, we derive
\begin{equation}
 \mathcal{R}(t,p^k, p^l) \leq p^{\nu} \sum_{u \shortmod{p^l}, p^\nu | u} p^{\nu} \leq p^{l + \nu},
\end{equation}
and in addition we have $\nu \leq l/2$.  Again, this is consistent with Lemma \ref{lemma:Rbound}.

\subsection{The case $k = l \geq 2$}
The case $k=l$ follows somewhat similar lines to the $k \neq l$ case, but there are some significant differences that require careful scrutiny.  We do not have a clean formula for $\widehat{S}$ analogous to \eqref{eq:ShatPrimePowerPrimePowerEvaluation}.

Performing some mild simplifications in \eqref{eq:ShatFormulaGeneral}, we obtain that $p^k | B_1 u$ and then
\begin{equation}
\widehat{S}(a,u,t,b,p^k,p^k) =  \mathop{\sum_{\substack{B_1, C_1 \shortmod{p^k} \\ (B_1, C_1, p^k) = 1}} 
 \sum_{\substack{ C_2 \shortmod{p^k} \\ (u, C_2, p^k) = 1}}
 }_{\substack{C_1 + C_2 \equiv - \frac{B_1 u}{p^k} \shortmod{p^{k}}
 \\ - Z_1 u \equiv t \shortmod{p^k}
 \\ Y_1 B_1 + Z_1 C_1 \equiv 1 \shortmod{p^k}
 \\ Y_2 u + Z_2 C_2 \equiv 1 \shortmod{p^k}
 }} e\Big(\frac{a B_1}{p^k} \Big)
 e\Big(\frac{ -b Z_2 B_1}{p^k} \Big).
\end{equation}
As before, let $V_{k_1}$ denote the subsum with $p^{k_1} || B_1$.  We have $p^{k_2} | u$, where $k_1 +k_2 = k$, but unlike the case $l > k \geq 2$, we cannot conclude that $p^{k_2} || u$.  We first estimate the cases with $k_1 = 0$ and $k_1 = k$.

We claim $V_0 = 0$. With these terms, we have $p \nmid B_1$, and so we may set $u=0$ (that is, we choose $u=0$ as the coset representative of $0 \pmod{p^k}$).  Then $p \nmid C_2$, and $C_1 \equiv - C_2 \pmod{p^k}$.  We also have $t=0$.  We may set $Y_2 = 0, Z_2 = \overline{C_2}$, and $Y_1 = \overline{B_1}$, $Z_1 = 0$.  With these evaluations, we derive
\begin{equation}
V_0 =  \sumstar_{B_1 \shortmod{p^k}} \thinspace \sumstar_{C_2 \shortmod{p^k}} 
e\Big(\frac{a B_1}{p^k} \Big)
 e\Big(\frac{ -b \overline{C_2} B_1}{p^k} \Big).
\end{equation}
The sum over $C_2$ vanishes since it is a Ramanujan sum with modulus $p^k$, $k \geq 2$ (see Lemma \ref{lemma:KloostermanVanishes}).

For $V_k$, we have $B_1 = 0$.  Then $p \nmid C_1$ and we set $Y_1 = 0$, $Z_1 = \overline{C_1}$.  Then \eqref{eq:C2congruence} becomes $C_2 \equiv - C_1 \pmod{p^k}$, so we have $p \nmid C_2$ and we are free to set $Y_2 =0$, $Z_2 = \overline{C_2}$.  Thus
\begin{equation}
V_k =  \sumstar_{\substack{C_1 \shortmod{p^k}  \\ - \overline{C_1} u \equiv t \shortmod{p^k}}} 
1.
\end{equation}
Since $u$ is uniquely determined from $C_1$, we have
\begin{equation}
\label{eq:Vkbound}
 \sum_{u \shortmod{p^k}} |V_k| = \phi(p^k). 
\end{equation}

Now consider $V_{k_1}$ with $1 \leq k_1 \leq k-1$.  Then $p | B_1$ and $p | u$ so $p \nmid C_1$, $p \nmid C_2$, and we set $Y_1 = Y_2 = 0$, $Z_i = \overline{C_i}$.  We write $B_1 = p^{k_1} R_1$.  
Suppose $\nu_p(t) =  \nu$, and write $t = p^{\nu} t''$.  Then we must have $\nu \geq k_2$, from the congruence $- \overline{C_1} u \equiv t \pmod{p^k}$,  and we can write $u = p^{\nu} u''$ instead of $u = p^{k_2} u'$.  
Then
\begin{equation}
\label{eq:Vk1formulaEqualprimepowers}
V_{k_1} = 
\sumstar_{R_1 \shortmod{p^{k_2}}}  
 \mathop{\sumstar_{\substack{C_1 \shortmod{p^k}}} 
 }_{\substack{ C_1 \equiv - u'' \overline{t''} \shortmod{p^{k-\nu}}
 }} e\Big(\frac{a R_1}{p^{k_2}} \Big)
 e\Big(\frac{ b R_1 \overline{(C_1 + R_1 p^{\nu - k_2} u'')}}{p^{k_2}} \Big).
\end{equation}

We claim that $V_{k_1} = 0$ if $k_1 < \nu < k$, as we now argue.
Observing that $k-\nu < k_2$ (since $k- \nu = k_1 + k_2 - \nu$), we can write 
$C_1 = -u'' \overline{t''} + p^{k-\nu} f_1$, with $f_1$ running mod $p^{\nu}$.  Then we can write $f_1 = f_2 + f_3 p^{\nu - k_1 -1}$ where $f_2$ runs mods $p^{\nu - k_1 - 1}$, and $f_3$ runs mod $p^{k_1 + 1}$.  Then we arrive at a sum over $f_3$ of the form
\begin{equation}
\sum_{f_3 \shortmod{p^{k_1 + 1}}} e\Big(\frac{\alpha \overline{(1 + \beta p^{k_2-1} f_3)}}{p^{k_2}} \Big) = \sum_{f_3 \shortmod{p^{k_1 + 1}}} e\Big(\frac{\alpha (1 - \beta p^{k_2-1} f_3)}{p^{k_2}} \Big),
\end{equation}
where $(\alpha \beta, p) = 1$, using $k_2 \geq 2$ which follows from $k_2 > k- \nu \geq 1$.  Since this sum over $f_3$ vanishes, this means $V_{k_1} = 0$.

If $\nu = k$, we have from \eqref{eq:Vk1formulaEqualprimepowers}, after changing variables $C_1 \rightarrow C_1 - R_1 p^{\nu - k_2} u''$, that
\begin{equation}
V_{k_1} = 
\sumstar_{R_1 \shortmod{p^{k_2}}}  \thinspace
 \sumstar_{\substack{C_1 \shortmod{p^k}}} 
  e\Big(\frac{a R_1}{p^{k_2}} \Big)
 e\Big(\frac{ b R_1 \overline{C_1}}{p^{k_2}} \Big) = p^{k_1} S(1,0;p^{k_2})^2 = 
 \begin{cases}
    0, \quad k_2 \geq 2, \\
    p^{k_1}, k_2 = 1.
 \end{cases}
\end{equation}
For $\nu =k$, we have
\begin{equation}
\label{eq:Vk1Boundsummed1}
\sum_{k_1=1}^{k-1} \sum_{u \shortmod{p^k}} |V_{k_1}| = p^{k-1} \delta(\nu =k).
\end{equation}

Now suppose $\nu \leq k_1 < k$.  This condition implies $k-\nu \geq k_2$, so the congruence $C_1 \equiv - u'' \overline{t''} \pmod{p^{k-\nu}}$ determines $C_1$ modulo $p^{k_2}$, and so \eqref{eq:Vk1formulaEqualprimepowers} simplifies as
\begin{equation}
\label{eq:Vk1formulaEqualprimepowersSimplified}
V_{k_1} = p^{\nu} 
\sumstar_{R_1 \shortmod{p^{k_2}}}  
  e\Big(\frac{a R_1}{p^{k_2}} \Big)
 e\Big(\frac{ b R_1 \overline{u''} \overline{(- \overline{t''} + R_1 p^{\nu - k_2} )}}{p^{k_2}} \Big).
\end{equation}
If, in addition, $\nu - k_2 \geq k_2$, then this is simply given by
\begin{equation}
V_{k_1} = p^{\nu} S(u'' a - t'' b, 0;p^{k_2}) \delta(k_2 \leq \nu/2).
\end{equation}
It follows easily that for these values of $k_2$ and $\nu$ that
\begin{equation}
\sum_{u \shortmod{p^k}} |V_{k_1}| \leq 2 p^{\nu} p^{k_2} \frac{p^{k-\nu}}{p^{k_2}} = 2p^{k},
\end{equation}
and so,
\begin{equation}
\label{eq:Vk1Boundsummed2}
 \sum_{1 \leq k_2 \leq \nu/2} \sum_{u \shortmod{p^k}} |V_{k_1}| \leq \nu p^k.
\end{equation}

On the other hand, if $\nu - k_2 < k_2$ (we continue to assume $\nu \leq k_1 < k$), then we can write $R_1 = f_1 + p^{2k_2 - \nu} f_2$, with $f_1$ mod $p^{2k_2 - \nu}$, and $f_2$ mod $p^{\nu - k_2}$.  Then \eqref{eq:Vk1formulaEqualprimepowersSimplified} becomes
\begin{equation}
V_{k_1} = p^{\nu} \sumstar_{f_1 \shortmod{p^{2k_2- \nu}}} \sum_{f_2 \shortmod{p^{\nu-k_2}}} 
e\Big(\frac{a (f_1 + p^{2k_2 - \nu} f_2)}{p^{k_2}} \Big)
 e\Big(\frac{ b \overline{u''} (f_1 + p^{2k_2 - \nu} f_2) \overline{(- \overline{t''} + f_1  p^{\nu - k_2} )}}{p^{k_2}} \Big).
\end{equation}
The inner sum over $f_2$ simplifies, and detects $u'' a \equiv b t'' \pmod{p^{\nu-k_2}}$.  Hence
\begin{equation}
 V_{k_1} = p^{2\nu-k_2} \sumstar_{f_1 \shortmod{p^{2k_2- \nu}}}
e\Big(\frac{a f_1}{p^{k_2}} \Big)
 e\Big(\frac{ b \overline{u''} f_1 \overline{(- \overline{t''} + f_1  p^{\nu - k_2} )}}{p^{k_2}} \Big) \delta(u''a \equiv t'' b \pmod{p^{\nu-k_2}}).
\end{equation}

Therefore, by a trivial bound on the $f_1$-sum, we have
\begin{equation}
|V_{k_1}| \leq  p^{2\nu - k_2} \phi(p^{2k_2 - \nu}) \delta(u''a \equiv t'' b \pmod{p^{\nu-k_2}}),
\end{equation}
which implies
\begin{equation}
\sum_{u \shortmod{p^k}} |V_{k_1}| \leq  p^{2\nu - k_2} \phi(p^{2k_2 - \nu}) \frac{p^{k-\nu}}{p^{\nu-k_2}} \leq \phi(p^{2k_2 + k-\nu}).
\end{equation}
Here we certainly have $k_2 \leq \nu$ (from the paragraph following \eqref{eq:Vkbound}), 
but we also have that $\nu$ is restricted by $\nu \leq \min(k-k_2, 2 k_2) \leq 2k/3$.  Therefore, these values of $k_1$ give
\begin{equation}
\label{eq:Vk1Boundsummed3}
 \sum_{k_1} \sum_{u \shortmod{p^k}} |V_{k_1}| \leq p^{k+\nu} \delta(\nu \leq 2k/3).
\end{equation}

Combining \eqref{eq:Vkbound}, \eqref{eq:Vk1Boundsummed1}, \eqref{eq:Vk1Boundsummed2}, and \eqref{eq:Vk1Boundsummed3}, we derive
\begin{equation}
 \mathcal{R}(t,p^k,p^k) \leq (k+1) p^k + p^{k+\nu} \delta(\nu \leq 2k/3).
\end{equation}

%
%

\section{Spectral summation formula}
The remaining sections of the paper contain the proof of Theorems \ref{thm:spectralsumlocal}.

We shall use the $GL_3$ Bruggeman-Kuznetsov formula in the form given by 
Blomer \cite{Blomer}.
Suppose that $T_1, T_2 \gg 1$, and consider the sum
\begin{equation}
\label{eq:spectralsumplain}
\sum_{\substack{\nu_1 = iT_1 + O(1) \\ \nu_2 = iT_2 + O(1)}} \frac{1}{R_j} \Big| \sum_{N/2 < n \leq N} a_n \lambda_j(n,1)
\Big|^2.
\end{equation}
As in \cite[(8.5)]{Blomer}, let
\begin{equation}
\label{eq:Fdef}
F(y_1, y_2) = \sqrt{T_1 T_2 (T_1 + T_2)} y_1^{i (\tau_1 + 2 \tau_2)}
y_2^{i(2 \tau_1 + \tau_2)} f(y_1) f(y_2),
\end{equation}
where $f$ is a fixed smooth, nonzero, non-negative function with support on
$[1,2]$.  Here $\tau_1, \tau_2$ are parameters satisfying $\tau_1 \asymp T_1$,
$\tau_2 \asymp T_2$.

By following the proof of \cite[Theorem 3]{Blomer}, we have that
\begin{equation}
\eqref{eq:spectralsumplain} \ll \sum_j \frac{1}{\| \phi_j \|^2} |\langle
\widetilde{W}_{\nu_1, \nu_2}, F \rangle|^2 \Big| \sum_{n \leq N} a_n
\lambda_j(n,1) \Big|^2 + (\text{cts}) =: S(T_1, T_2, N),
\end{equation}
where $\widetilde{W}_{\nu_1, \nu_2}$ is a completed Whittaker function associated to the Maass form $\phi_j$,
and (cts)
represents the non-negative continuous spectrum contribution.  

The Bruggeman-Kuznetsov formula in the form of \cite[Proposition 4]{Blomer} says
\begin{equation}
S(T_1, T_2, N)  = \sum_{n} \sum_m \overline{a_n} a_m \Big(\Sigma_1 + \Sigma_{2a}
+ \Sigma_{2b} + \Sigma_3 \Big),
\end{equation}
where
\begin{equation}
 \Sigma_1 = \delta_{m=n} \| F \|^2,
\end{equation}
and $\Sigma_{2a}, \Sigma_{2b}$, and $\Sigma_{3}$ are sums involving $GL_3$
Kloosterman sums.  With $F$ defined by \eqref{eq:Fdef}, then $\|F\|^2 \asymp T_1
T_2 (T_1 + T_2)$, which is the mass of the spectral ball to account for the
diagonal terms.  It turns out that we do not need to analyze $\Sigma_{2a}$ and
$\Sigma_{2b}$, so we omit their definitions (the interested reader may find them defined in \cite[(8.2)]{Blomer}).

Here
\begin{equation}
 \Sigma_{3} = \sum_{\epsilon_1, \epsilon_2 = \pm 1} \sum_{\substack{D_1, D_2 }}
\frac{S(1, \epsilon_1 m, \epsilon_2 n, 1, D_1, D_2)}{D_1 D_2}
\mathcal{J}_{\epsilon_1, \epsilon_2}  \Big(\frac{\sqrt{mD_1}}{D_2},
\frac{\sqrt{nD_2}}{D_1} \Big),
\end{equation}
where with shorthand $x_3' = x_1 x_2 - x_3$,
\begin{multline}
 \mathcal{J}_{\epsilon_1, \epsilon_2}(A_1, A_2) =
 (A_1 A_2)^{-2} \int_0^{\infty} \int_0^{\infty} \intR \intR \intR e(-\epsilon_1
A_1 x_1 y_1 - \epsilon_2 A_2 x_2 y_2)
 \\
 e\Big(-\frac{A_2}{y_2} \frac{x_1 x_3 + x_2}{1 + x_2^2 + x_3^2}\Big)
 e\Big(-\frac{A_1}{y_1} \frac{x_2 x_3' + x_1}{1 + x_1^2 + x_3'^2}\Big)
 \\
 F\Big(\frac{A_2}{y_2} \frac{\sqrt{1 + x_1^2 + x_3'^2}}{1+x_2^2 + x_3^2},
\frac{A_1}{y_1} \frac{\sqrt{1 + x_2^2 + x_3^2}}{1+x_1^2 + x_3'^2}\Big)
 \overline{F}(A_1 y_1, A_2 y_2) dx_1 dx_2 dx_3 \frac{dy_1 dy_2}{y_1 y_2}.
\end{multline}
One pleasant feature of this integral expression is that the variables are practically separated, and the kernel function is easily bounded uniformly in all parameters.  For comparison, the formulas of Buttcane \cite[Theorem 2]{Buttcane2} also directly separate the variables and only require a $2$-fold integral, which should in principle be more efficient.  However the tradeoff is that the kernel function is not as easy to bound, requiring one to work on multiple scales.  Furthermore, the weight function depends on $\epsilon_1, \epsilon_2$ in a non-trivial way, leading to further case analysis.

Blomer (see \cite[p.722]{Blomer}) showed that
\begin{equation}
 \sum_{m,n} a_m \overline{a_n} (|\Sigma_{2a}| + |\Sigma_{2b}|) \ll
N^{\varepsilon} (T_1 + T_2)^{-100} \sum_n |a_n|^2,
\end{equation}
which means these terms are practically negligible.  This estimate arises
because the weight function on the sum of Kloosterman sums side is very small
for these terms.
Taken together, this shows
\begin{mylemma}
 We have
 \begin{multline}
  \sum_{\substack{\nu_1 = iT_1 + O(1) \\ \nu_2 = iT_2 + O(1)}} \frac{1}{R_j} \Big| \sum_{N/2 < n \leq N} a_n \lambda_j(n,1)
\Big|^2 \ll T_1 T_2 (T_1 + T_2) \sum_n |a_n|^2
+ \sum_{m,n} a_m \overline{a_n} \Sigma_{3}
\\
 + \frac{N^{\varepsilon}}{(T_1 +
T_2)^{100}} \sum_{n} |a_n|^2.
 \end{multline}
\end{mylemma}

\section{Manipulations of $\Sigma_3$}
Our next step is to perform some elementary manipulations to $\Sigma_3$ in order to prepare it for the use of Theorem \ref{thm:BilinearKloosterman}.
Note that we can change variables $y_1 \rightarrow y_1/A_1$ and $y_2 \rightarrow
y_2/A_2$ and use the definitions $\xi_1 = 1 + x_1^2 + x_3'^2$, $\xi_2 = 1 +
x_2^2 + x_3^2$, to give
\begin{multline}
 \mathcal{J}_{\epsilon_1, \epsilon_2}(A_1, A_2) =
 (A_1 A_2)^{-2} \int_0^{\infty} \int_0^{\infty} \intR \intR \intR e(-\epsilon_1
x_1 y_1 - \epsilon_2  x_2 y_2)
 \\
 e\Big(-\frac{A_2^2}{y_2} \frac{x_1 x_3 + x_2}{1 + x_2^2 + x_3^2}\Big)
 e\Big(-\frac{A_1^2}{y_1} \frac{x_2 x_3' + x_1}{1 + x_1^2 + x_3'^2}\Big)
 F\Big(\frac{A_2^2}{y_2} \frac{\xi_1^{1/2}}{\xi_2}, \frac{A_1^2}{y_1}
\frac{\xi_2^{1/2}}{\xi_1}\Big)
 \overline{F}( y_1, y_2) dx_1 dx_2 dx_3 \frac{dy_1 dy_2}{y_1 y_2}.
\end{multline}
Our next step is to insert the definition \eqref{eq:Fdef}, and re-arrange the resulting expression.  For later use, it may be helpful to note
that for our values of $A_1$ and $A_2$ that
\begin{equation}
 (A_2^2)^{i(\tau_1 + 2\tau_2)}
 (A_1^2)^{i(2 \tau_1 + \tau_2)} = m^{i(2\tau_1 + \tau_2)} n^{i (\tau_1 + 2
\tau_2)} D_1^{-3i \tau_2} D_2^{-3i \tau_1}.
\end{equation}
In this way, we obtain
\begin{multline}
\label{eq:sigma3intermediateexpression}
 \sum_{n,m} \overline{a_n} a_m \Sigma_3 =
 \sum_{\epsilon_1, \epsilon_2 = \pm 1}
 \intR \intR \intR \xi_1^{-\frac{3}{2} i \tau_1 } \xi_2^{-\frac32 i \tau_2} dx_1
dx_2 dx_3
 \\
 \sum_{\substack{D_1, D_2 }}
 \sum_{n,m} \frac{T_1 T_2 (T_1 +T_2)}{\frac{mn}{D_1 D_2}} n^{i(\tau_1 + 2 \tau_2)} m^{i(2\tau_1 + \tau_2)} \overline{a_n} a_m
 \frac{S(1, \epsilon_1 m, \epsilon_2 n, 1, D_1, D_2)}{D_1^{1 + 3i
\tau_2} D_2^{1 + 3i\tau_1}}
  \\
\Big[ \int_0^{\infty} y_2^{-3i( \tau_1 + \tau_2)}
 e( - \epsilon_2  x_2 y_2) e\Big(-\frac{n D_2}{y_2 D_1^2} \frac{x_1 x_3 + x_2}{1
+ x_2^2 + x_3^2}\Big)
  f\Big(\frac{n D_2}{y_2 D_1^2} \frac{\xi_1^{1/2}}{\xi_2}\Big)
   f(y_2) \frac{dy_2}{y_2}  \Big]
     \\
\Big[ \int_0^{\infty}
 y_1^{-3i(\tau_1 + \tau_2)}
 e(-\epsilon_1  x_1 y_1)
 e\Big(-\frac{m D_1}{y_1 D_2^2} \frac{x_2 x_3' + x_1}{1 + x_1^2 + x_3'^2}\Big)
f\Big(\frac{m D_1}{y_1 D_2^2} \frac{\xi_2^{1/2}}{\xi_1}\Big)
 f(y_1)  \frac{dy_1}{y_1 } \Big].
\end{multline}

Now let us restrict to $D_1 \asymp X_1$, and $D_2 \asymp X_2$, and sum over these dyadic values of $X_1, X_2$ at the end.  Since $m,n \asymp N$, if we  let $(N/n) a_n = a_n'$, then $a_n \asymp a_n'$.  The support on $f$ constrains the $x$-variables
into a certain region $V$ of $\mr^3$ that has measure $\ll (A_1 A_2)^{2 +
\varepsilon}= (\frac{mn}{D_1 D_2})^{1+\frac{\varepsilon}{2}}$, by \cite[Lemma 4]{Blomer}.  Furthermore, $V$ is independent of $m,n, y_1, y_2, D_1, D_2$ (it depends on $N, X_1, X_2$).
Write $|V| = \frac{N^2}{X_1 X_2}$, so that the measure of $V$ is at most $|V| N^{\varepsilon}$.

Blomer's bound \cite[(8.9)]{Blomer} shows that $\mathcal{J}_{\epsilon_1, \epsilon_2}(A_1, A_2)$ is very small for $A_1^{4/3} A_2^{2/3} \ll (T_1 + T_2)^{1-\varepsilon}$, or $A_1^{2/3} A_2^{4/3} \ll (T_1 + T_2)^{1-\varepsilon}$.  This means that we may assume
\begin{equation}
X_1, X_2 \ll \frac{N}{(T_1 + T_2)^{1-\varepsilon}}.
\end{equation}
Remark.  In \cite[Lemma 9]{BlomerButtcane}, Blomer and Buttcane have shown, using Buttcane's Mellin-Barnes integral representations, that the $X_i$ can be truncated earlier, in the special case $T_1 \asymp T_2 \asymp T$, to $X_i \ll T^{-2+\varepsilon} N$.  

Then we have
\begin{equation}
\label{eq:Sigma3toU}
 \Big| \sum_{n,m} \overline{a_n} a_m \Sigma_3 \Big| \ll 
 \sum_{\epsilon_1, \epsilon_2 = \pm 1}
 \sum_{X_1, X_2 \text{ dyadic}} \frac{N^{\varepsilon} T_1 T_2 (T_1 +T_2)}{X_1 X_2}
 \int_{x_1, x_2, x_3 \in V} \frac{dx_1 dx_2 dx_3}{|V|} |U|,
\end{equation}
where $U = U(\gamma, \epsilon_1, \epsilon_2, X_1, X_2, x_1, x_2, x_3)$ is defined by
\begin{multline}
\label{eq:Udef}
U=
 \sum_{\substack{D_1 \asymp X_1 \\ D_2 \asymp X_2 }}
 \sum_{n,m} 
   b_n'' a_m''
 S(1, \epsilon_1 m, \epsilon_2 n, 1, D_1, D_2) \gamma_{D_1, D_2}
  \\
\Big[ \int_0^{\infty} y_2^{-3i( \tau_1 + \tau_2)}
 e( - \epsilon_2  x_2 y_2) e\Big(-\frac{n D_2}{y_2 D_1^2} \frac{x_1 x_3 + x_2}{1
+ x_2^2 + x_3^2}\Big)
  f\Big(\frac{n D_2}{y_2 D_1^2} \frac{\xi_1^{1/2}}{\xi_2}\Big)
   f(y_2) \frac{dy_2}{y_2}  \Big]
     \\
\Big[ \int_0^{\infty}
 y_1^{-3i(\tau_1 + \tau_2)}
 e(-\epsilon_1  x_1 y_1)
 e\Big(-\frac{m D_1}{y_1 D_2^2} \frac{x_2 x_3' + x_1}{1 + x_1^2 + x_3'^2}\Big)
f\Big(\frac{m D_1}{y_1 D_2^2} \frac{\xi_2^{1/2}}{\xi_1}\Big)
 f(y_1)  \frac{dy_1}{y_1 } \Big],
\end{multline}
for some sequence $\gamma_{D_1, D_2}$ with $|\gamma_{D_1, D_2}| \leq 1$.
Here we have used the shorthand $a_m'' = a_m' m^{i(2 \tau_1 + \tau_2)}$, $b_n'' = \overline{a_n'} n^{i(\tau_1 + 2 \tau_2)}$.

\begin{myprop}
\label{prop:Ubound}
We have
\begin{equation}
\label{eq:Ubound}
 |U| \ll X_1 X_2  \Big(X_1^2 + \frac{N}{T_1 + T_2} \Big)^{1/2} \Big(X_2^2 + \frac{N}{T_1 + T_2}\Big)^{1/2} (N T_1 T_2)^{\varepsilon} \sum_{n \leq N} |a_n|^2.
\end{equation}
The estimate is uniform in terms of $\gamma, X_1, X_2, x_1, x_2, x_3$.
\end{myprop}
Assuming Proposition \ref{prop:Ubound}, we may quickly show the following variant of Theorem \ref{thm:spectralsumlocal}:
\begin{equation}
\label{eq:mainthmVariant}
  \sum_{\substack{\nu_1 = iT_1 + O(1) \\ \nu_2 = iT_2 + O(1)}} \frac{1}{R_j} \Big| \sum_{ n \leq N} a_n \lambda_j(n,1)
\Big|^2 \ll
\Big(T_1 T_2 (T_1 + T_2) +   T_1 T_2 \frac{N^2}{T_1 + T_2} \Big)^{1+\varepsilon}
\sum_{n \leq N} |a_n|^2,
 \end{equation}
as we now explain.
By inserting \eqref{eq:Ubound} into \eqref{eq:Sigma3toU}, we obtain
\begin{multline}
\label{eq:UBoundAppliedToSigma3}
\Big| \sum_{n,m} \overline{a_n} a_m \Sigma_3 \Big| 
\ll  \sum_{\substack{X_1, X_2 
\ll \frac{N}{(T_1 + T_2)^{1-\varepsilon}} \\ \text{dyadic} }} \frac{ T_1 T_2 (T_1 + T_2) (T_1 T_2 N)^{\varepsilon}}{X_1 X_2}
\\
 X_1 X_2 \Big(X_1^2 + \frac{N}{T_1 + T_2} \Big)^{1/2} \Big(X_2^2 + \frac{N}{T_1 + T_2}\Big)^{1/2} \sum_{n \leq N} |a_n|^2,
\end{multline}
plus a small error term from the truncation on $X_1, X_2$.
By a direct calculation, this gives
\begin{equation}
 \Big|\sum_{n,m} \overline{a_n} a_m \Sigma_3 \Big|
\ll T_1 T_2 (T_1 + T_2) (T_1 T_2 N)^{\varepsilon} \Big(\frac{N}{T_1 + T_2}\Big)^{2}
\sum_n |a_n|^2,
\end{equation}
which proves \eqref{eq:mainthmVariant}.

\begin{proof}[Proof of Proposition \ref{prop:Ubound}]
The main difficulty in the proof is exploiting cancellation in the $y_1, y_2$ integrals.  For point of reference, if we apply Corollary \ref{coro:Sbound}
directly to \eqref{eq:Udef}, trivially integrating over $y_1$ and $y_2$, we obtain
\begin{equation}
\label{eq:UboundTrivial}
 |U| \ll X_1 X_2  (X_1^2 + N )^{1/2} (X_2^2 + N )^{1/2} (N T_1 T_2)^{\varepsilon} \sum_{n \leq N} |a_n|^2.
\end{equation}
One easily observes that if $T_1 + T_2 \ll N^{\varepsilon}$, then \eqref{eq:UboundTrivial} implies \eqref{eq:Ubound}, so for the rest of the proof we assume
\begin{equation}
\label{eq:T1andT2nottoosmall}
 T_1 + T_2 \gg N^{\varepsilon}.
\end{equation}

We will first show the bound \eqref{eq:Ubound} under the assumptions
\begin{equation}
\label{eq:smallxicondition}
|x_1|, |x_2| \leq \delta (T_1 + T_2),
\end{equation}
where $\delta > 0$ is some small but fixed number (certainly $1/1000$ suffices for the proof).  
Let
\begin{equation}
Y_1 = \frac{X_2^2}{X_1} \frac{1 + x_1^2 + x_3'^2}{x_2 x_3' + x_1}, \quad
\text{and} \quad Y_2 = \frac{X_1^2}{X_2} \frac{1 + x_2^2 + x_3^2}{x_1 x_3 +
x_2}.
\end{equation}
With this definition, we have that the $y_1$-integral takes the form
\begin{equation}
 \int h(y_1) e^{i \phi_1(y_1)} dy_1, \quad \phi_1(y_1) = c_1 (T_1 + T_2) \log
y_1 + c_2 x_1 y_1 + c_3 \frac{N}{y_1 Y_1},
\end{equation}
where each $c_i \asymp 1$ and $h$ is a weight function with bounded
derivatives.  Under the assumption \eqref{eq:smallxicondition}, repeated integration by parts (see \cite[Lemma 8.1]{BKY}) shows the integral is smaller than an arbitrarily large negative power of $\max(T_1 + T_2, \frac{N}{Y_1})$ (and hence, using \eqref{eq:T1andT2nottoosmall}, an arbitrarily large power of $T_1T_2 N$), unless $\frac{N}{Y_1} \asymp (T_1 + T_2)$.  If \eqref{eq:smallxicondition} does not hold then there is potentially cancellation between the first two terms in the phase in which case this argument breaks down. 

A similar argument holds for $y_2$ also.  Thus
we may assume
\begin{equation}
\label{eq:Yibound}
Y_1, Y_2 \ll \frac{N}{T_1 + T_2}.
\end{equation}
Moving the integrals to the outside, we derive
\begin{multline}
 |U| \ll \int_1^2 \int_1^2 \Big|
\sum_{\substack{D_1 \asymp X_1 \\ D_2 \asymp X_2 }}
 \sum_{n,m} 
   b_n'' a_m''
 S(1, \epsilon_1 m, \epsilon_2 n, 1, D_1, D_2) \gamma_{D_1, D_2}
  \\
  e\Big(-\frac{n D_2}{y_2 D_1^2} \frac{x_1 x_3 + x_2}{1
+ x_2^2 + x_3^2}\Big)
  f\Big(\frac{n D_2}{y_2 D_1^2} \frac{\xi_1^{1/2}}{\xi_2}\Big) 
 e\Big(-\frac{m D_1}{y_1 D_2^2} \frac{x_2 x_3' + x_1}{1 + x_1^2 + x_3'^2}\Big)
f\Big(\frac{m D_1}{y_1 D_2^2} \frac{\xi_2^{1/2}}{\xi_1}\Big)
 \Big| \frac{dy_1}{y_1} \frac{dy_2}{y_2}.
\end{multline}
 Now we can change variables $y_1 \rightarrow y_1^{-1} \frac{D_1}{X_1} \frac{X_2^2}{D_2^2}$, and $y_2 \rightarrow y_2^{-1} \frac{D_2}{X_2} \frac{X_1^2}{D_1^2}$, giving that 
 \begin{multline}
 \label{eq:UboundMiddleofProof}
 |U| \ll \int_{y_1 \asymp 1} \int_{y_2 \asymp 1} \Big|
\sum_{\substack{D_1 \asymp X_1 \\ D_2 \asymp X_2 }}
 \sum_{n,m} 
   b_n'' a_m''
 S(1, \epsilon_1 m, \epsilon_2 n, 1, D_1, D_2) \gamma_{D_1, D_2}
 \\
 e\Big(-\frac{m y_1}{Y_1}\Big) e\Big(-\frac{n y_2}{Y_2}\Big) 
  f\Big(\frac{n y_2 X_2}{X_1^2} \frac{\xi_1^{1/2}}{\xi_2}\Big) 
f\Big(\frac{m y_1 X_1}{X_2^2} \frac{\xi_2^{1/2}}{\xi_1}\Big)
 \Big| \frac{dy_1}{y_1} \frac{dy_2}{y_2}.
\end{multline}
Now we can apply Mellin inversion to $f(\frac{ n y_2 X_2}{X_1^2} \frac{\xi_1^{1/2}}{\xi_2})$ (and the other $f$), showing now
\begin{multline}
 |U| \ll \intR  \frac{1}{1 + r_1^2} \intR  \frac{1}{1 + r_2^2} \int_{y_1 \asymp 1} \int_{y_2 \asymp 1} 
 \Big|
\sum_{\substack{D_1 \asymp X_1 \\ D_2 \asymp X_2 }}
\gamma_{D_1, D_2}
\\
 \sum_{n,m} 
   b_n'' n^{ir_2} a_m'' m^{ir_1}
 S(1, \epsilon_1 m, \epsilon_2 n, 1, D_1, D_2)  e\Big(-\frac{m y_1}{Y_1}\Big) e\Big(-\frac{n y_2}{Y_2}\Big) 
 \Big| \frac{dy_1}{y_1} \frac{dy_2}{y_2} dr_1 dr_2.
\end{multline}
Remark.  The $r_1$ and $r_2$ integrals are practically harmless because our bound will be in terms of the $L^2$ norms of the sequences $(b_n'')$ and $(a_m'')$, which are then independent of $r_1, r_2$.

At this point we can apply Theorem \ref{thm:BilinearKloosterman} (see also Remark \ref{remark:bilinearwithKloostermanTheoremVariant}), showing
\begin{multline}
\label{eq:UBoundAfterBilinearTheoremApplied}
 |U| \ll (X_1 X_2)^{1+\varepsilon}
 \Big[
 \intR \frac{1}{1 + r_2^2}
 \\
 \int_{y_2 \asymp 1} 
 \sum_{q \leq \min(X_1, X_2)} \sum_{d_1| q} \frac{d_1}{q} \sum_{\substack{c \leq \frac{X_1}{q} \\ (c,q) = 1}} \thinspace \sumstar_{t \shortmod{c}} 
 \Big|  
  \sum_{(n,q) = d_1} b_n'' n^{ir_2} e\Big(\frac{tn}{c}\Big) e\Big(-\frac{ny_2}{Y_2}\Big)
 \Big|^2  dy_2 dr_2 \Big]^{1/2} 
 [\dots ]^{1/2},
\end{multline}
with $[\dots]$ representing a similar term.  
Using the hybrid large sieve (Lemma \ref{lemma:variantsieve}) shows that the first expression in brackets is bounded by
\begin{equation}
 \sum_{q \leq \min(X_1, X_2)} \sum_{d_1 | q} \frac{d_1}{q} \Big(\frac{X_1^2}{q^2} + \frac{Y_2}{d_1} \Big) \sum_{(n,q) = d_1} |\beta_n|^2 \ll (X_1 X_2 )^{\varepsilon} \Big(X_1^2 + \frac{N}{T_1 + T_2} \Big) \sum_{n \leq N} |b_n|^2.
\end{equation}
The second expression in brackets in \eqref{eq:UBoundAfterBilinearTheoremApplied} is bounded in a similar way, which completes the proof under the assumption \eqref{eq:smallxicondition}.

Now we show how to modify the proof in case \eqref{eq:smallxicondition} does not hold.  Say that $|x_1| \geq \delta (T_1 + T_2)$, and $|x_2| \leq \delta (T_1 + T_2)$.  The $y_2$-analysis is unchanged while in
\eqref{eq:Udef}, we change variables $y_1 \rightarrow y_1 \frac{m}{N}$.  Following the calculations above, in place of \eqref{eq:UboundMiddleofProof}, we obtain
\begin{multline}
 \label{eq:UboundMiddleofProof2}
 |U| \ll \int_{y_1 \asymp 1} \int_{y_2 \asymp 1} \Big|
\sum_{\substack{D_1 \asymp X_1 \\ D_2 \asymp X_2 }}
 \sum_{n,m} 
   b_n'' a_m'''
 S(1, \epsilon_1 m, \epsilon_2 n, 1, D_1, D_2) \gamma_{D_1, D_2}
 \\
 e\Big(-\frac{m y_1 x_1}{N}\Big) e\Big(-\frac{n y_2}{Y_2}\Big) 
  f\Big(\frac{n y_2 X_2}{X_1^2} \frac{\xi_1^{1/2}}{\xi_2}\Big) 
f\Big(\frac{m y_1}{N}\Big)
 \Big| \frac{dy_1}{y_1} \frac{dy_2}{y_2},
\end{multline}
where $|a_m'''| = |a_m''|$.  This has the same essential form as \eqref{eq:UboundMiddleofProof} but with $Y_1$ replaced by $\frac{N}{|x_1|} \ll \frac{N}{T_1 + T_2}$.  Thus we arrive at the same bound in this case.  By symmetry, the same bound holds in case $|x_1| \leq \delta (T_1 + T_2)$ and $|x_2| \geq \delta (T_1 + T_2)$.  A simple modification covers the case $|x_1|, |x_2| \geq \delta(T_1 + T_2)$, where we apply the change of variables in both $y_1, y_2$.
\end{proof}

\section{Proof of Theorem \ref{thm:spectralsumlocal}}
If the $X_i$ are large, then we can obtain an improved version of Proposition \ref{prop:Ubound}, namely
\begin{myprop}
\label{prop:UboundWithH}
 We have
 \begin{multline}
  |U| \ll \Big[(X_1 H_2 + X_2 H_1)  \Big(X_1^2 + \frac{N}{T_1 + T_2} \Big)^{1/2} \Big(X_2^2 + \frac{N}{T_1 + T_2} \Big)^{1/2}  
  \\
  + \frac{(X_1 X_2)^{3/2} N}{H_1} + \frac{(X_1 X_2)^{3/2} N}{H_2}
  \Big] (X_1 X_2)^{\varepsilon} \sum_{n \leq N} |a_n|^2.
 \end{multline}
\end{myprop}
Here the proof is identical to that of Proposition \ref{prop:Ubound} except at \eqref{eq:UBoundAfterBilinearTheoremApplied} we apply Theorem \ref{thm:BilinearKloostermanHversion} instead of Theorem \ref{thm:BilinearKloosterman}, so we omit the details.

We continue with bounding \eqref{eq:Sigma3toU}. We shall use the bound implied by
\eqref{eq:UBoundAppliedToSigma3} for certain ranges of $X_i$.  Specifically, for the values of $X_1$ with $X_1^2 \leq \frac{N}{T_1 + T_2}$, the bound \eqref{eq:UBoundAppliedToSigma3} simplifies as
\begin{multline}
 \ll \sum_{\substack{X_1^2 \ll \frac{N}{T_1 + T_2} \\ X_2 \ll \frac{N}{(T_1 + T_2)^{1-\varepsilon}}}} T_1 T_2 (T_1 + T_2) \Big(\frac{N}{T_1 + T_2} \Big)^{1/2} \Big(X_2^2 + \frac{N}{T_1 + T_2}\Big)^{1/2} (T_1 T_2 )^{\varepsilon} \sum_{n \leq N} |a_n|^2 
 \\
 \ll T_1 T_2 \frac{N^{3/2}}{(T_1 + T_2)^{1/2}} (T_1 T_2 )^{\varepsilon} \sum_{n \leq N} |a_n|^2,
\end{multline}
which is stronger than required for Theorem \ref{thm:spectralsumlocal}.
By symmetry, the same bound holds if $X_2^2 \leq \frac{N}{T_1 + T_2}$.  For the complementary terms with $X_1^2 > \frac{N}{T_1 + T_2}$ and $X_2^2 > \frac{N}{T_1 + T_2}$, Proposition \ref{prop:UboundWithH} simplifies to give
\begin{equation}
 |U| \ll (X_1 X_2)\Big[(X_1 H_2 + X_2 H_1) + \frac{(X_1 X_2)^{1/2} N}{H_1} + \frac{(X_1 X_2)^{1/2} N}{H_2}
  \Big] (X_1 X_2 )^{\varepsilon} \sum_{n \leq N} |a_n|^2.
\end{equation}
The optimal choice is $H_1 = N^{1/2} X_1^{1/4} X_2^{-1/4}$, $H_2 = N^{1/2} X_1^{-1/4} X_2^{1/4}$, and gives
\begin{equation}
 |U| \ll (X_1 X_2) N^{1/2} (X_1^{3/4} X_2^{1/4} + X_1^{1/4} X_2^{3/4})  (X_1 X_2)^{\varepsilon} \sum_{n \leq N} |a_n|^2.
\end{equation}
The contribution of these terms to \eqref{eq:UBoundAppliedToSigma3} is then seen to be
\begin{equation}
 \ll T_1 T_2 (T_1 + T_2) N^{1/2} \Big(\frac{N}{T_1 + T_2} \Big) (T_1 T_2 )^{\varepsilon} \sum_{n \leq N} |a_n|^2 \ll T_1 T_2 N^{3/2} (T_1 T_2 )^{\varepsilon} \sum_{n \leq N} |a_n|^2.
\end{equation}
This is precisely what is required for Theorem \ref{thm:spectralsumlocal}.

\end{document}